\documentclass{article}%
\usepackage{amsmath}
\usepackage{amsfonts}
\usepackage{amssymb}
\usepackage{graphicx}%
\providecommand{\U}[1]{\protect \rule{.1in}{.1in}}
\newtheorem{theorem}{Theorem}

\newtheorem{definition}[theorem]{Definition}
\newtheorem{example}[theorem]{Example}

\newtheorem{lemma}[theorem]{Lemma}

\newtheorem{proposition}[theorem]{Proposition}
\newtheorem{remark}[theorem]{Remark}

\newenvironment{proof}[1][Proof]{\noindent \textbf{#1.} }{\  $\Box$}
\begin{document}

\title{Extended Conditional $G$-Expectations and Related Stopping Times}
\author{Mingshang Hu \thanks{School of Mathematics, Shandong University,
humingshang@sdu.edu.cn. Research supported by the National Natural
Science Foundation of China (11201262 and 11101242)} \and Shige
Peng\thanks{School of Mathematics and Qilu Institute of Finance,
Shandong University, peng@sdu.edu.cn, Hu and Peng's research was
partially supported by NSF of China No. 10921101; and by the 111
Project No. B12023}}

\maketitle
\date{}

\begin{abstract}
In this paper we extend the definition of time conditional
$G$-expectations $\mathbb{\hat{E}}_{t}[\cdot]$ to a larger domain on
which the dynamical consistency still holds. In fact we can
consistently define, by taking the limit, the time conditional
expectations for each random variable $X$ which is the downward
limit (resp. upward limit) of a monotone sequence $\{X_{i}\}$ in
$L_{G}^{1}(\Omega)$. To accomplish this procedure, some careful
analysis is needed. Moreover, we give a suitable definition of stopping times and obtain the optional stopping theorem. We also provide some basic and interesting
properties for the extended conditional $G$-expectations.

\end{abstract}


\textbf{MSC-classification}: 60H10, 60H30

\section{Introduction}

A typical $G$-expectation is a sublinear expectation defined on a
linear space of random variables $L_{G}^{1}(\Omega)$ which is the
completion of the linear space $L_{ip}(\Omega)$ of Lipschitz
cylinder functions on some $d$-dimensional continuous path space
$\Omega=C_0^d[0,\infty)$. A sublinear expectation
$\mathbb{\hat{E}}{\normalsize :\ }L_{ip}(\Omega)\mapsto\mathbb{R}$,
called
$G$-expectation, is constructed under which the canonical path $B_{t}%
(\omega):=\omega_{t}$ is a $G$-Brownian motion, namely, it has independent and
stable increments under the the sublinear expectation $\mathbb{\hat{E}}$. The
completion is under the natural Banach norm defined by $\left\Vert
X\right\Vert :=\mathbb{\hat{E}}{\normalsize [|X|]}$ for $X$ in $L_{ip}%
(\Omega)$. A $1$-dimensional $G$-Brownian motion, for $d=1$, is
characterized by a sublinear monotone function $G$ defined by%
\[
G(x):=\frac{1}{2}\mathbb{\hat{E}}{\normalsize [x}B_{1}^{2}]:\mathbb{R}%
\mapsto\mathbb{R}\text{.}%
\]
In general each sublinear and monotone function $G$ corresponds to a unique $G$-expectation
$\mathbb{\hat{E}}[\cdot]=\mathbb{\hat{E}}_{G}[\cdot]$ under which the
canonical process is the corresponding $G$-Brownian motion. $G\geq\bar{G}$ if
and only if $\mathbb{\hat{E}}_{G}\geq\mathbb{\hat{E}}_{\bar{G}}$. A typical
situation is: $G$ dominates $\bar{G}(x):=\frac{x}{2}$. In this case $\bar{G}%
$-Brownian motion is in fact the classical standard Brownian under a linear
expectation $\mathbb{\hat{E}}_{\bar{G}}=E_{P}$, which the linear expectation
is induced by a classical Wiener measure $P$ on $(\Omega,\mathcal{B}%
{\normalsize (\Omega))}$. In this case $G\geq\bar{G}$ implies that the Banach
norm $\mathbb{\hat{E}}[|\cdot|]$ is stronger than $E_{P}[|\cdot|]$,
consequently $L_{G}^{1}(\Omega)$ becomes a  subspace of the classical
$L^{1}(\Omega)$ space under the Wiener measure $P$.

An important advantage of the $G$-framework is that the time conditional
expectation is well defined on $L_{G}^{1}(\Omega)$ with time consistency, thus
the notion of nonlinear martingales can be naturally introduced. A
generalization of stochastic calculus of It\^{o}'s type is also established.

One of main problems of this $G$-expectation framework is that the
space of random variables $L_{G}^{p}(\Omega)$ is still not big
enough to contain some interesting random variables. For example,
given a continuous, or right continuous, stochastic process
$(X_{t})_{t\geq0}$ such that $X_{t}\in L_{G}^{1}(\Omega_{t})$, and
the exit time $\tau$ of $X$ from some domain, the random variable
$X_{\tau}$ may fail to be in $L_{G}^{p}(\Omega)$.   Many research papers are devoted to solve  this very attractive problem, see, among others,  \cite{GPP-2},  \cite{L-P}, \cite{NH}, \cite{Song-11} and  \cite {NZ}, but many things are still to be understood.

In this paper we will  attack this problem extend the definition
of time conditional $G$-expectations to a space of random variables
larger than $L_{G}^{1}(\Omega)$ on which the dynamical consistency
still holds. The main idea is quite simple: we consistently define,
by taking the limit, the time conditional expectations for each
random variable $X$ which is the downward
limit (respectively upward limit) of a monotone sequence $\{X_{n}%
\}_{n=1}^{\infty}$ in $L_{G}^{1}(\Omega)$. To accomplish this
procedure we need some careful analysis. We also provide some basic
and interesting properties for the extended conditional
$G$-expectation.

This paper is organized as follows: in the next section we recall
some basic results of the $G$-framework on $G$-expectation. In
Section~3 we take three steps to extend $G$-expectation to larger
domains. In Section 3.1, we define
$L_{G}^{1^{\ast}}(\Omega)$ space which is the downward extension of $L_{G}%
^{1}(\Omega)$; then, in Section~3.2, define
$L_{G}^{1_{\ast}^{\ast}}(\Omega)$ which is a \textquotedblleft
upward extension\textquotedblright\ of $L_{G}^{1^{\ast}}(\Omega)$;
and finally in Section 3.3, we extend furthermore the space
$\bar{L}_{G}^{1_{\ast}^{\ast}}(\Omega)$ by taking the completion
under the norm induced by our $G$-expectation. In Section 3.4 we
discuss some important random elements in our new framework but may not
 be in the classical $G$-expectation space
$L_{G}^{1}(\Omega)$. Moreover, we give the definition of stopping times and obtain
the optional stopping theorem. Section 4 shows that a more general nonlinear
expectations dominated by a sublinear $G$-expectation can also have
its extension.

We recall that \cite{NH} also extended the conditional $G$-expectation to a very general class of functions. We  limit ourself to treat the case where the random variables are still within   the class of Borel measurable functions.

\section{Preliminaries}

We review some basic notions and results of $G$-expectation and the related
spaces of random variables. The readers may refer to \cite{DHP11},
\cite{HJPS}, \cite{P07a}, \cite{P07b}, \cite{P08a}, \cite{P08b}, \cite{P10}
for more details.

\subsection{$G$-expectations}

Let $\Omega$ be a given set and let $\mathcal{H}$ be a vector lattice of real
valued functions defined on $\Omega$, namely $c\in \mathcal{H}$ for each
constant $c$ and $|X|\in \mathcal{H}$ if $X\in \mathcal{H}$. We further suppose
that if $X_{1},\ldots,X_{n}\in \mathcal{H}$, then $\varphi(X_{1},\cdots
,X_{n})\in \mathcal{H}$ for each $\varphi \in C_{b.Lip}(\mathbb{R}^{n})$, where
$C_{b.Lip}(\mathbb{R}^{n})$ denotes the space of bounded and Lipschitz
functions. $\mathcal{H}$ is considered as the space of random variables.

\begin{definition}
\label{def2.1} A sublinear expectation $\mathbb{\hat{E}}$ on $\mathcal{H}$ is
a functional $\mathbb{\hat{E}}:\mathcal{H}\rightarrow \mathbb{R}$ satisfying
the following properties: for all $X,Y\in \mathcal{H}$, we have

\begin{description}
\item[(a)] Monotonicity: If $X\geq Y$ then $\mathbb{\hat{E}}[X]\geq
\mathbb{\hat{E}}[Y]$;

\item[(b)] Constant preservation: $\mathbb{\hat{E}}[c]=c$;

\item[(c)] Sub-additivity: $\mathbb{\hat{E}}[X+Y]\leq \mathbb{\hat{E}%
}[X]+\mathbb{\hat{E}}[Y]$;

\item[(d)] Positive homogeneity: $\mathbb{\hat{E}}[\lambda X]=\lambda
\mathbb{\hat{E}}[X]$ for each $\lambda \geq0$.
\end{description}

The triple $(\Omega,\mathcal{H},\mathbb{\hat{E}})$ is called a sublinear
expectation space.
\end{definition}

\begin{remark}
If the inequality in (c) becomes equality, then $\mathbb{\hat{E}}$ is called a
linear expectation.
\end{remark}

\begin{definition}
\label{def2.2} Let $X_{1}$ and $X_{2}$ be two $n$-dimensional random vectors
defined respectively in sublinear expectation spaces $(\Omega_{1}%
,\mathcal{H}_{1},\mathbb{\hat{E}}_{1})$ and $(\Omega_{2},\mathcal{H}%
_{2},\mathbb{\hat{E}}_{2})$. They are called identically distributed, denoted
by $X_{1}\overset{d}{=}X_{2}$, if $\mathbb{\hat{E}}_{1}[\varphi(X_{1}%
)]=\mathbb{\hat{E}}_{2}[\varphi(X_{2})]$, for all$\  \varphi \in C_{b.Lip}%
(\mathbb{R}^{n})$.
\end{definition}

\begin{definition}
\label{def2.3} In a sublinear expectation space $(\Omega,\mathcal{H}%
,\mathbb{\hat{E}})$, a random vector $Y=(Y_{1},\cdot \cdot \cdot,Y_{n})$,
$Y_{i}\in \mathcal{H}$, is said to be independent of another random vector
$X=(X_{1},\cdot \cdot \cdot,X_{m})$, $X_{i}\in \mathcal{H}$ under $\mathbb{\hat
{E}}[\cdot]$, denoted by $Y\bot X$, if for every test function $\varphi \in
C_{b.Lip}(\mathbb{R}^{m}\times \mathbb{R}^{n})$ we have $\mathbb{\hat{E}%
}[\varphi(X,Y)]=\mathbb{\hat{E}}[\mathbb{\hat{E}}[\varphi(x,Y)]_{x=X}]$.
\end{definition}

\begin{definition}
\label{def2.4} ($G$-normal distribution) A $d$-dimensional random vector
$X=(X_{1},\cdot \cdot \cdot,X_{d})$ in a sublinear expectation space
$(\Omega,\mathcal{H},\mathbb{\hat{E}})$ is called $G$-normally distributed if
$\mathbb{\hat{E}}[|X|^{3}]<\infty$ and for each $a,b\geq0$
\[
aX+b\bar{X}\overset{d}{=}\sqrt{a^{2}+b^{2}}X,
\]
where $\bar{X}$ is an independent copy of $X$, i.e., $\bar{X}\overset{d}{=}X$
and $\bar{X}\bot X$. Here the letter $G$ denotes the function
\[
G(A):=\frac{1}{2}\mathbb{\hat{E}}[\langle AX,X\rangle]:\mathbb{S}%
_{d}\rightarrow \mathbb{R},
\]
where $\mathbb{S}_{d}$ denotes the collection of $d\times d$ symmetric matrices.
\end{definition}

Peng \cite{P08b} showed that $X=(X_{1},\cdot \cdot \cdot,X_{d})$ is $G$-normally
distributed if and only if for each $\varphi \in C_{b.Lip}(\mathbb{R}^{d})$,
$u(t,x):=\mathbb{\hat{E}}[\varphi(x+\sqrt{t}X)]$, $(t,x)\in \lbrack
0,\infty)\times \mathbb{R}^{d}$, is the solution of the following $G$-heat
equation:%
\[
\partial_{t}u-G(D_{x}^{2}u)=0,\ u(0,x)=\varphi(x).
\]

The function $G(\cdot):\mathbb{S}_{d}\rightarrow \mathbb{R}$ is a monotonic,
sublinear mapping on $\mathbb{S}_{d}$, which implies that there exists a
bounded, convex and closed subset $\Sigma \subset \mathbb{S}_{d}^{+}$ such that
\[
G(A)=\frac{1}{2}\sup_{B\in \Sigma}\mathrm{tr}[AB],
\]
where $\mathbb{S}_{d}^{+}$ denotes the collection of nonnegative elements in
$\mathbb{S}_{d}$.

\begin{definition}
\label{def2.5} i) Let $\Omega=C_{0}^{d}(\mathbb{R}^{+})$ be the space of all
$\mathbb{R}^{d}$-valued continuous paths $(\omega_{t})_{t\in \mathbb{R}^{+}}$,
with $\omega_{0}=0$, equipped with the disance%
\[
\rho(\omega^{1},\omega^{2}):=\sum_{i=1}^{\infty}2^{-i}[(\max_{t\in \lbrack
0,i]}|\omega_{t}^{1}-\omega_{t}^{2}|)\wedge1].
\]
The canonical process is defined by $B_{t}(\omega)=\omega_{t}$, $t\in
\lbrack0,\infty)$, for $\omega \in \Omega$. Set
\[
L_{ip}(\Omega):=\{ \varphi(B_{t_{1}},...,B_{t_{n}}):n\geq1,t_{1},...,t_{n}%
\in \lbrack0,\infty),\varphi \in C_{b.Lip}(\mathbb{R}^{d\times n})\}.
\]
Let $G:\mathbb{S}_{d}\rightarrow \mathbb{R}$ be a given monotonic and sublinear
function. $G$-expectation is a sublinear expectation on $L_{ip}(\Omega)$
defined by
\[
\mathbb{\hat{E}}[X]=\mathbb{\tilde{E}}[\varphi(\sqrt{t_{1}-t_{0}}\xi_{1}%
,\cdot \cdot \cdot,\sqrt{t_{m}-t_{m-1}}\xi_{m})],
\]
for all $X=\varphi(B_{t_{1}}-B_{t_{0}},B_{t_{2}}-B_{t_{1}},\cdot \cdot
\cdot,B_{t_{m}}-B_{t_{m-1}})$ with $0\leq t_{0}<t_{1}<\cdots<t_{m}<\infty$,
where $\xi_{1},\cdot \cdot \cdot,\xi_{n}$ are identically distributed
$d$-dimensional $G$-normally distributed random vectors in a sublinear
expectation space $(\tilde{\Omega},\tilde{\mathcal{H}},\mathbb{\tilde{E}})$
such that $\xi_{i+1}$ is independent of $(\xi_{1},\cdot \cdot \cdot,\xi_{i})$
for every $i=1,\cdot \cdot \cdot,m-1$. The corresponding canonical process
$B_{t}=(B_{t}^{i})_{i=1}^{d}$ is called a $G$-Brownian motion.

ii) Set $\Omega_{t}=\{ \omega_{\cdot \wedge t}:\omega \in \Omega \}$ for $t\geq0$.
For each $\xi=\varphi(B_{t_{1}}-B_{t_{0}},B_{t_{2}}-B_{t_{1}},\cdot \cdot
\cdot,B_{t_{m}}-B_{t_{m-1}})$, the conditional $G$-expectation of $\xi$ under
$\Omega_{t_{i}}$ is defined by
\[
\mathbb{\hat{E}}_{t_{i}}[\varphi(B_{t_{1}}-B_{t_{0}},B_{t_{2}}-B_{t_{1}}%
,\cdot \cdot \cdot,B_{t_{m}}-B_{t_{m-1}})]
\]%
\[
=\tilde{\varphi}(B_{t_{1}}-B_{t_{0}},B_{t_{2}}-B_{t_{1}},\cdot \cdot
\cdot,B_{t_{i}}-B_{t_{i-1}}),
\]
where
\[
\tilde{\varphi}(x_{1},\cdot \cdot \cdot,x_{i})=\mathbb{\hat{E}}[\varphi
(x_{1},\cdot \cdot \cdot,x_{i},B_{t_{i+1}}-B_{t_{i}},\cdot \cdot \cdot,B_{t_{m}%
}-B_{t_{m-1}})].
\]

\end{definition}

For each fixed $T\geq0$, we set%
\[
L_{ip}(\Omega_{T}):=\{ \varphi(B_{t_{1}},...,B_{t_{n}}):n\geq1,t_{1}%
,...,t_{n}\in \lbrack0,T],\varphi \in C_{b.Lip}(\mathbb{R}^{d\times n})\}.
\]
It is clear that $L_{ip}(\Omega_{T_{1}})\subset L_{ip}(\Omega_{T_{2}})\subset
L_{ip}(\Omega)$ for $T_{1}<T_{2}$. We denote by $L_{G}^{p}(\Omega)$ and
$L_{G}^{p}(\Omega_{T})$, $p\geq1$, the completion of $L_{ip}(\Omega)$ and
$L_{ip}(\Omega_{T})$ under the norm $\Vert \xi \Vert_{p}=(\mathbb{\hat{E}}%
[|\xi|^{p}])^{1/p}$. It is easy to verify that $L_{G}^{p_{1}}(\Omega)\subset
L_{G}^{p_{2}}(\Omega)$ for $p_{1}\geq p_{2}\geq1$.

The $G$-expectation $\mathbb{\hat{E}}[\cdot]$ can be continuously extended to
a sublinear expectation on $(\Omega,L_{G}^{1}(\Omega))$ still denoted by
$\mathbb{\hat{E}}[\cdot]$. For each given $t\geq0$, the conditional
$G$-expectation $\mathbb{\hat{E}}_{t}[\cdot]:L_{ip}(\Omega)\rightarrow
L_{ip}(\Omega_{t})$ can be also extended as a mapping $\mathbb{\hat{E}}%
_{t}[\cdot]:L_{G}^{1}(\Omega)\rightarrow L_{G}^{1}(\Omega_{t})$ and satisfies
the following properties:

\begin{description}
\item[(i)] If $X$, $Y\in L_{G}^{1}(\Omega)$, $X\geq Y$, then $\mathbb{\hat{E}%
}_{t}[X]\geq \mathbb{\hat{E}}_{t}[Y]$;

\item[(ii)] If $X\in L_{G}^{1}(\Omega_{t})$, $Y\in L_{G}^{1}(\Omega)$, then
$\mathbb{\hat{E}}_{t}[X+Y]=X+\mathbb{\hat{E}}_{t}[Y]$;

\item[(iii)] If $X$, $Y\in L_{G}^{1}(\Omega)$, then $\mathbb{\hat{E}}%
_{t}[X+Y]\leq \mathbb{\hat{E}}_{t}[X]+\mathbb{\hat{E}}_{t}[Y]$;

\item[(iv)] If $X\in L_{G}^{1}(\Omega_{t})$ is bounded, $Y\in L_{G}^{1}%
(\Omega)$, then $\mathbb{\hat{E}}_{t}[XY]=X^{+}\mathbb{\hat{E}}_{t}%
[Y]+X^{-}\mathbb{\hat{E}}_{t}[-Y]$;

\item[(v)] If $X\in L_{G}^{1}(\Omega)$, then $\mathbb{\hat{E}}_{s}%
[\mathbb{\hat{E}}_{t}[X]]=\mathbb{\hat{E}}_{s\wedge t}[X]$, in paricular,
$\mathbb{\hat{E}}[\mathbb{\hat{E}}_{t}[X]]=\mathbb{\hat{E}}[X]$.
\end{description}

For each fixed $\mathbf{a}\in \mathbb{R}^{d}$, $B_{t}^{\mathbf{a}}%
=\langle \mathbf{a},B_{t}\rangle$ is a $1$-dimensional $G_{\mathbf{a}}%
$-Brownian motion, where $G_{\mathbf{a}}(\alpha)=\frac{1}{2}(\sigma
_{\mathbf{aa}^{T}}^{2}\alpha^{+}-\sigma_{-\mathbf{aa}^{T}}^{2}\alpha^{-})$,
$\sigma_{\mathbf{aa}^{T}}^{2}=2G(\mathbf{aa}^{T})$, $\sigma_{-\mathbf{aa}^{T}%
}^{2}=-2G(-\mathbf{aa}^{T})$. Let $\pi_{t}^{N}=\{t_{0}^{N},\cdots,t_{N}^{N}%
\}$, $N=1,2,\cdots$, be a sequence of partitions of $[0,t]$ such that $\mu
(\pi_{t}^{N})=\max \{|t_{i+1}^{N}-t_{i}^{N}|:i=0,\cdots,N-1\} \rightarrow0$,
the quadratic variation process of $B^{\mathbf{a}}$ is defined by%
\[
\langle B^{\mathbf{a}}\rangle_{t}=L_{G}^{2}-\lim_{\mu(\pi_{t}^{N}%
)\rightarrow0}\sum_{j=0}^{N-1}(B_{t_{j+1}^{N}}^{\mathbf{a}}-B_{t_{j}^{N}%
}^{\mathbf{a}})^{2}.
\]
For each fixed $\mathbf{a}$, $\mathbf{\bar{a}}\in \mathbb{R}^{d}$, the mutual
variation process of $B^{\mathbf{a}}$ and $B^{\mathbf{\bar{a}}}$ is defined by%
\[
\langle B^{\mathbf{a}},B^{\mathbf{\bar{a}}}\rangle_{t}=\frac{1}{4}[\langle
B^{\mathbf{a}+\mathbf{\bar{a}}}\rangle_{t}-\langle B^{\mathbf{a}%
-\mathbf{\bar{a}}}\rangle_{t}].
\]

\begin{definition}
\label{def2.6} Let $M_{G}^{0}(0,T)$ be the collection of processes in the
following form: for a given partition $\{t_{0},\cdot \cdot \cdot,t_{N}\}=\pi
_{T}$ of $[0,T]$,
\[
\eta_{t}(\omega)=\sum_{j=0}^{N-1}\xi_{j}(\omega)I_{[t_{j},t_{j+1})}(t),
\]
where $\xi_{j}\in L_{ip}(\Omega_{t_{j}})$, $j=0,1,2,\cdot \cdot \cdot,N-1$. For
$p\geq1$ and $\eta \in M_{G}^{0}(0,T)$, let $\Vert \eta \Vert_{H_{G}^{p}}=\{
\mathbb{\hat{E}}[(\int_{0}^{T}|\eta_{s}|^{2}ds)^{p/2}]\}^{1/p}$, $\Vert
\eta \Vert_{M_{G}^{p}}=\{ \mathbb{\hat{E}}[\int_{0}^{T}|\eta_{s}|^{p}%
ds]\}^{1/p}$. We denote by $H_{G}^{p}(0,T)$, $M_{G}^{p}(0,T)$ the completions
of $M_{G}^{0}(0,T)$ under the norms $\Vert \cdot \Vert_{H_{G}^{p}}$, $\Vert
\cdot \Vert_{M_{G}^{p}}$ respectively.
\end{definition}

For each $\eta_{t}=\sum_{j=0}^{N-1}\xi_{j}I_{[t_{j},t_{j+1})}(t)\in M_{G}%
^{0}(0,T)$, define $\int_{0}^{T}\eta_{t}dt=\sum_{j=0}^{N-1}\xi_{j}%
(t_{j+1}-t_{j})$, $\int_{0}^{T}\eta_{t}d\langle B^{\mathbf{a}}\rangle_{t}%
=\sum_{j=0}^{N-1}\xi_{j}(\langle B^{\mathbf{a}}\rangle_{t_{j+1}}-\langle
B^{\mathbf{a}}\rangle_{t_{j}})$ and $\int_{0}^{T}\eta_{t}dB_{t}^{\mathbf{a}%
}=\sum_{j=0}^{N-1}\xi_{j}(B_{t_{j+1}}^{\mathbf{a}}-B_{t_{j}}^{\mathbf{a}})$.
Under the norm $\Vert \cdot \Vert_{M_{G}^{1}}$, $\int_{0}^{T}\eta_{t}dt$ and
$\int_{0}^{T}\eta_{t}d\langle B^{\mathbf{a}}\rangle_{t}$ can be extended to
$\eta \in M_{G}^{1}(0,T)$. $\int_{0}^{T}\eta_{t}dB_{t}^{\mathbf{a}}$ can be
extended to $\eta \in H_{G}^{1}(0,T)$ under the norm $\Vert \cdot \Vert
_{H_{G}^{1}}$.

For each fixed $n\in \mathbb{N}$, we set $L_{ip}(\Omega;\mathbb{R}^{n})=\{
\xi=(\xi_{1},\ldots,\xi_{n}):\xi_{i}\in L_{ip}(\Omega)$, $i\leq n\}$.
Similarly, we can define $L_{ip}(\Omega_{t};\mathbb{R}^{n})$, $L_{G}%
^{p}(\Omega;\mathbb{R}^{n})$, $L_{G}^{p}(\Omega_{t};\mathbb{R}^{n})$,
$H_{G}^{p}(0,T;\mathbb{R}^{n})$ and $M_{G}^{p}(0,T;\mathbb{R}^{n})$.

\subsection{$G$-capacities}

Let $(\Omega,L_{ip}(\Omega),\mathbb{\hat{E}}[\cdot])$ be the $G$-expectation
space. The filtration is
\begin{align*}
\mathcal{F}_{t}  &  =\sigma \{B_{s}:s\leq t\},\  \  \ t\geq0,\\
\  \mathcal{F}  &  =\vee_{t\geq0}\mathcal{F}_{t}=\mathcal{B}(\Omega).
\end{align*}
We denote by $\mathcal{M}$ the set of all probability measures on
$(\Omega,\mathcal{B}(\Omega))$.

\begin{theorem}
\label{the2.7} (\cite{DHP11,HP09}) There exists a weakly compact set
$\mathcal{P}\subset \mathcal{M}$ such that
\[
\mathbb{\hat{E}}[X]=\sup_{P\in \mathcal{P}}E_{P}[X]\  \  \text{for \ each}\ X\in
L_{ip}(\Omega).
\]
$\mathcal{P}$ is called a set that represents $\mathbb{\hat{E}}$.
\end{theorem}

Let $\mathcal{P}$ be a weakly compact set that represents $\mathbb{\hat{E}}$.
For this $\mathcal{P}$, we define the capacity%
\[
c(A)=\sup_{P\in \mathcal{P}}P(A)\text{ for each }A\in \mathcal{B}(\Omega).
\]
One can verify the following proposition.

\begin{proposition}
\label{pro2.8} $c(\cdot)$ satisfies the following properties:

\begin{description}
\item[(1)] $A\subset B\Longrightarrow c(A)\leq c(B)$;

\item[(2)] $c(\cup A_{n})\leq \sum c(A_{n})$;

\item[(3)] $A_{n}\uparrow A\Longrightarrow c(A_{n})\uparrow c(A)$;

\item[(4)] $F_{n}$ closed set, $F_{n}\downarrow F\Longrightarrow
c(F_{n})\downarrow c(F)$;

\item[(5)] For each $A\in \mathcal{B}(\Omega)$, $c(A)=\sup \{c(K):K$ compact
$K\subset A\}$.
\end{description}
\end{proposition}

A set $A\subset \Omega$ is polar if $c(A)=0$. A property holds
\textquotedblleft quasi-surely\textquotedblright \ (q.s. for short) if it holds
outside a polar set.

We set%
\[
L^{0}(\Omega):=\{X:\Omega \rightarrow \lbrack-\infty,\infty]\text{ and }X\text{
is }\mathcal{F}\text{-measurable}\},
\]%
\[
\mathcal{L}(\Omega):=\{X\in L^{0}(\Omega):E_{P}[X]\text{ exists for each }%
P\in \mathcal{P}\}.
\]
We extend $G$-expectation $\mathbb{\hat{E}}$ to $\mathcal{L}(\Omega)$ and
still denote it by $\mathbb{\hat{E}}$. For each $X\in \mathcal{L}(\Omega)$, we
define%
\[
\mathbb{\hat{E}}[X]=\sup_{P\in \mathcal{P}}E_{P}[X].
\]
If $X$, $Y\in \mathcal{L}(\Omega)$ such that $X=Y$ q.s., then $E_{P}%
[X]=E_{P}[Y]$ for each $P\in \mathcal{P}$. In the following, we do not
distinguish two random variables $X$ and $Y$ if $X=Y$ q.s.. We set%
\[
\mathbb{L}^{p}(\Omega):=\{X\in L^{0}(\Omega):\mathbb{\hat{E}}[|X|^{p}%
]<\infty \} \  \text{for}\ p\geq1.
\]
Obviously, $\mathbb{L}^{p}(\Omega)\subset \mathcal{L}(\Omega)$. For $p\geq1$,
$\mathbb{L}^{p}(\Omega)$ is a Banach space under the norm $(\mathbb{\hat{E}%
}[|\cdot|^{p}])^{1/p}$. Similarly, we can define $L^{0}(\Omega_{t})$,
$\mathcal{L}(\Omega_{t})$ and $\mathbb{L}^{p}(\Omega_{t})$.

A function $X:\Omega \rightarrow \lbrack-\infty,\infty]$ is called
quasi-continuous (q.c.) if for each $\varepsilon>0$, there exists a closed set
$F$ with $c(F^{c})<\varepsilon$ such that $X|_{F}$ is continuous. We say that
$Y:\Omega \rightarrow \lbrack-\infty,\infty]$ has a quasi-continuous version if
there exists a quasi-continuous function $X:\Omega \rightarrow \lbrack
-\infty,\infty]$ with $Y=X$ q.s..

\begin{theorem}
\label{the2.9}(\cite{DHP11,HP09}) For each $p\geq1$,%
\[
L_{G}^{p}(\Omega)=\{X\in \mathbb{L}^{p}(\Omega):X\text{ has a q.c. version,
}\lim_{n\rightarrow \infty}\mathbb{\hat{E}}[|X|^{p}I_{\{|X|>n\}}]=0\}.
\]

\end{theorem}

\begin{theorem}
\label{the2.10}(\cite{DHP11,P10}) $\mathbb{\hat{E}}[\cdot]$ satisfies the
following properties:

\begin{description}
\item[(a)] If $X_{n}\in \mathcal{L}(\Omega)\uparrow X\in \mathcal{L}(\Omega)$
q.s. and $-\mathbb{\hat{E}}[-X_{1}]>-\infty$, then $\mathbb{\hat{E}}%
[X_{n}]\uparrow \mathbb{\hat{E}}[X]$;

\item[(b)] If $P_{n}\in \mathcal{P}$ converges weakly to $P\in \mathcal{P}$,
then $E_{P_{n}}[X]\rightarrow E_{P}[X]$ for each $X\in L_{G}^{1}(\Omega)$;

\item[(c)] If $X_{n}\in L_{G}^{1}(\Omega)\downarrow X\in \mathcal{L}(\Omega)$
q.s., then $\mathbb{\hat{E}}[X_{n}]\downarrow \mathbb{\hat{E}}[X]$.
\end{description}
\end{theorem}

\begin{theorem}
\label{the2.11}(\cite{HJPS}) We have

\begin{description}
\item[(1)] For fixed $A\in \mathcal{F}_{t}$, if $\xi_{1}$, $\xi_{2}\in
L_{G}^{1}(\Omega)$ such that $\xi_{1}I_{A}=\xi_{2}I_{A}$ q.s., then
$\mathbb{\hat{E}}_{t}[\xi_{1}]I_{A}=\mathbb{\hat{E}}_{t}[\xi_{2}]I_{A}$ q.s.;

\item[(2)] Let $(A_{i})_{i=1}^{n}$ be an $\mathcal{F}_{t}$-partition of
$\Omega$. Then for $\xi_{i}\in L_{G}^{1}(\Omega)$, $i\leq n$, we have
$\mathbb{\hat{E}}[\sum_{i=1}^{n}\xi_{i}I_{A_{i}}]=\mathbb{\hat{E}}[\sum
_{i=1}^{n}\mathbb{\hat{E}}_{t}[\xi_{i}]I_{A_{i}}]$.
\end{description}
\end{theorem}

We can define the following convergence.

\begin{itemize}
\item quasi sure (q.s.) convergence:\ $X_{n}\overset{\text{q.s.}%
}{\longrightarrow}X$ means \thinspace$c(\left \{  X_{n}\not \rightarrow
X\right \}  =0$;

\item convergence in capacity $c$: $X_{n}\overset{\text{c}}{\longrightarrow}X$
means \thinspace$\lim_{n\rightarrow \infty}c(\{|X_{n}-X|\geq \varepsilon \})=0$
for each $\varepsilon>0$;

\item $\mathbb{L}^{p}$-convergence: Let $\left \{  X_{n}\right \}
_{n=1}^{\infty}\subset \mathbb{L}^{p}(\Omega)$, $X\in \mathbb{L}^{p}(\Omega)$
for $p\geq1$. $X_{n}\overset{\mathbb{L}^{p}}{\rightarrow}X$ means
\thinspace$\lim_{n\rightarrow \infty}\mathbb{\hat{E}}[|X_{n}-X|^{p}]=0$.
\end{itemize}

\begin{theorem}
\label{the2.12} We have

\begin{description}
\item[(1)] If $X_{n}\overset{\mathbb{L}^{p}}{\rightarrow}X$ for $p\geq1$, then
$X_{n}\overset{\text{c}}{\longrightarrow}X$;

\item[(2)] If $X_{n}\overset{\text{c}}{\longrightarrow}X$, then there exists a
subsequence $\left \{  X_{n_{k}}\right \}  _{k=1}^{\infty}$ such that $X_{n_{k}%
}\overset{\text{q.s.}}{\longrightarrow}X$.
\end{description}
\end{theorem}

\begin{proof}
We omit the proof which is similar to classical case.
\end{proof}

We set%
\[
\mathcal{P}_{\max}=\{P\in \mathcal{M}:E_{P}[X]\leq \mathbb{\hat{E}}[X]\text{ for
each }X\in L_{ip}(\Omega)\}.
\]
It is easy to check that $\mathcal{P}\subset \mathcal{P}_{\max}$ and
$\mathcal{P}_{\max}$ represents $\mathbb{\hat{E}}$. Similarly, we define%
\[
\tilde{c}(A)=\sup_{P\in \mathcal{P}_{\max}}P(A)\text{ for each }A\in
\mathcal{B}(\Omega).
\]
For each $X\in L^{0}(\Omega)$ such that $E_{P}[X]$ exists for each
$P\in \mathcal{P}_{\max}$, we define%
\[
\mathbb{\tilde{E}}[X]=\sup_{P\in \mathcal{P}_{\max}}E_{P}[X].
\]

\begin{proposition}
\label{pro2.13} For each $A\in \mathcal{B}(\Omega)$, we have $c(A)=\tilde
{c}(A)$.
\end{proposition}

\begin{proof}
It follows from the definition of $L_{G}^{1}(\Omega)$ that $\mathbb{\hat{E}%
}[X]=$ $\mathbb{\tilde{E}}[X]$ for each $X\in L_{G}^{1}(\Omega)$. By Theorem
\ref{the2.9}, we get $\mathbb{\hat{E}}[X]=$ $\mathbb{\tilde{E}}[X]$ for each
$X\in C_{b}(\Omega)$. For each fixed closed set $F\subset \Omega$, we can
choose $X_{n}\in C_{b}(\Omega)$ such that $X_{n}\downarrow I_{F}$. Thus, by
Theorem \ref{the2.10}, we obtain $\mathbb{\hat{E}}[I_{F}]=\mathbb{\tilde{E}%
}[I_{F}]$, which implies $c(F)=\tilde{c}(F)$. By Proposition \ref{pro2.8}, we
get $c(A)=\tilde{c}(A)$ for each $A\in \mathcal{B}(\Omega)$.
\end{proof}

\begin{remark}
\label{rem2.14} Let $\mathcal{P}_{1}$ and $\mathcal{P}_{2}$ be two weakly
compact sets that represent $\mathbb{\hat{E}}$. From the above proposition, we
can deduce that%
\[
\sup_{P\in \mathcal{P}_{1}}P(A)=\sup_{P\in \mathcal{P}_{2}}P(A)\text{ for each
}A\in \mathcal{B}(\Omega).
\]

\end{remark}

\subsection{The representation of conditional $G$-expectations}

Let $(\Omega,L_{ip}(\Omega),\mathbb{\hat{E}}[\cdot])$ be the $G$-expectation
space. For each given $t\geq0$, we set $B_{s}^{t}=B_{s}-B_{t}$ and%
\[
L_{ip}(\Omega^{t})=\{ \varphi(B_{t_{1}}^{t},...,B_{t_{n}}^{t}):n\geq
1,t_{1},...,t_{n}\in \lbrack t,\infty),\varphi \in C_{b.Lip}(\mathbb{R}^{d\times
n})\}.
\]
In this subsection, we suppose that $\mathcal{P}=\mathcal{P}_{\max}$. For each
fixed $t\geq0$ and $P\in \mathcal{P}$, we define%
\[
\mathcal{P}(t,P)=\{Q\in \mathcal{P}:E_{Q}[X]=E_{P}[X],\forall X\in
L_{ip}(\Omega_{t})\}.
\]

The following representation of the
conditional $G$-expectation was obtained in \cite{STZ}.

\begin{theorem}
\label{the2.15} For each $X\in L_{G}^{1}(\Omega)$, we have, for each
$P\in \mathcal{P}$,%
\[
\mathbb{\hat{E}}_{t}[X]=\underset{Q\in \mathcal{P}(t,P)}{ess\sup}^{P}%
E_{Q}[X|\mathcal{F}_{t}]\text{ \  \ }P\text{-a.s.}.
\]

\end{theorem}

Here we give a new proof which still holds for general case. For this we first
consider the representation of the conditional $G$-expectation for $X\in
L_{ip}(\Omega)$.

\begin{lemma}
\label{le2.16} For each fixed $Q\in \mathcal{P}$, we have $E_{Q}[X|\mathcal{F}%
_{t}]\leq \mathbb{\hat{E}}_{t}[X]$, $Q$-a.s., for each $X\in L_{ip}(\Omega)$.
\end{lemma}

\begin{proof}
Step 1. We first assert that $E_{Q}[X|\mathcal{F}_{t}]\leq \mathbb{\hat{E}}%
[X]$, $Q$-a.s., for each $X\in L_{ip}(\Omega^{t})$. Otherwise, we can choose a
$\xi \in L_{ip}(\Omega^{t})$ such that $Q(A)>0$, where $A=\{E_{Q}%
[\xi|\mathcal{F}_{t}]>\mathbb{\hat{E}}[\xi]\}$. Thus%
\begin{align*}
E_{Q}[I_{A}\xi+I_{A^{c}}\mathbb{\hat{E}}[\xi]]  &  =E_{Q}[I_{A}E_{Q}%
[\xi|\mathcal{F}_{t}]]+Q(A^{c})\mathbb{\hat{E}}[\xi]\\
&  >E_{Q}[I_{A}\mathbb{\hat{E}}[\xi]]+Q(A^{c})\mathbb{\hat{E}}[\xi]\\
&  =\mathbb{\hat{E}}[\xi].
\end{align*}
On the other hand, $E_{Q}[I_{A}\xi+I_{A^{c}}\mathbb{\hat{E}}[\xi
]]\leq \mathbb{\hat{E}}[I_{A}\xi+I_{A^{c}}\mathbb{\hat{E}}[\xi]]$. By Theorem
\ref{the2.11}, we have%
\[
\mathbb{\hat{E}}[I_{A}\xi+I_{A^{c}}\mathbb{\hat{E}}[\xi]]=\mathbb{\hat{E}%
}[I_{A}\mathbb{\hat{E}}_{t}[\xi]+I_{A^{c}}\mathbb{\hat{E}}[\xi]]=\mathbb{\hat
{E}}[\xi],
\]
which implies a contradiction. Thus $E_{Q}[X|\mathcal{F}_{t}]\leq
\mathbb{\hat{E}}[X]$, $Q$-a.s., for each $X\in L_{ip}(\Omega^{t})$.

Step 2. For each $X\in L_{ip}(\Omega)$, there exist $t_{1}<t_{2}<\cdots<t_{n}$
with $t=t_{k}$ and $\Phi \in C_{b.Lip}(\mathbb{R}^{d\times n})$ such that
$X=\Phi(B_{t_{1}},B_{t_{2}}-B_{t_{1}},\ldots,B_{t_{n}}-B_{t_{n-1}})$. Thus%
\[
E_{Q}[X|\mathcal{F}_{t}]=\Psi_{1}(B_{t_{1}},\cdots,B_{t_{k}}-B_{t_{k-1}%
})\text{ \  \ }Q\text{-a.s.},
\]%
\[
\mathbb{\hat{E}}_{t}[X]=\Psi_{2}(B_{t_{1}},\cdots,B_{t_{k}}-B_{t_{k-1}})
\]
where $\Psi_{1}(x_{1},\cdots,x_{k})=E_{Q}[\Phi(x_{1},\cdots,x_{k},B_{t_{k+1}%
}-B_{t_{k}},\ldots,B_{t_{n}}-B_{t_{n-1}})|\mathcal{F}_{t}]$, $\Psi_{2}%
(x_{1},\cdots,x_{k})=\mathbb{\hat{E}}[\Phi(x_{1},\cdots,x_{k},B_{t_{k+1}%
}-B_{t_{k}},\ldots,B_{t_{n}}-B_{t_{n-1}})]$. By Step 1, we know $\Psi
_{1}(x_{1},\cdots,x_{k})\leq \Psi_{2}(x_{1},\cdots,x_{k})$, $Q$-a.s.. It is
easy to verify that $|\Psi_{i}(x)-\Psi_{i}(x^{\prime})|\leq L_{\Phi
}|x-x^{\prime}|$ for each $x$, $x^{\prime}\in \mathbb{R}^{d\times k}$, $i=1,2$,
where $L_{\Phi}$ is the Lipschitz constant of $\Phi$. Thus $E_{Q}%
[X|\mathcal{F}_{t}]\leq \mathbb{\hat{E}}_{t}[X]$, $Q$-a.s..
\end{proof}

\begin{lemma}
\label{le2.17} For each $P$, $Q\in \mathcal{P}$ and $t\geq0$, set
$E[\varphi(\xi,\eta)]=E_{P}[E_{Q}[\varphi(x,\eta)]_{x=\xi}]$ for each $m$,
$n\in \mathbb{N}$, $\xi \in L_{ip}(\Omega_{t};\mathbb{R}^{m})$, $\eta \in
L_{ip}(\Omega^{t};\mathbb{R}^{n})$ and $\varphi \in C_{b.Lip}(\mathbb{R}%
^{m+n})$. Then there exists a unique $Q^{\ast}\in \mathcal{P}(t,P)$ such that
$E[X]=E_{Q^{\ast}}[X]$ for each $X\in L_{ip}(\Omega)$ and $E_{Q^{\ast}%
}[\varphi(\xi,\eta)|\mathcal{F}_{t}]=E_{Q}[\varphi(x,\eta)]_{x=\xi}$, $P$-a.s..
\end{lemma}

\begin{proof}
It is easy to check that $E[\cdot]$ is a linear expectation on $L_{ip}%
(\Omega)$ and $E[X]\leq \mathbb{\hat{E}}[X]$ for each $X\in L_{ip}(\Omega)$. By
Theorem \ref{the2.10} and the Daniell-Stone theorem, there exists a unique
$Q^{\ast}\in \mathcal{M}$ such that $E[X]=E_{Q^{\ast}}[X]$ for each $X\in
L_{ip}(\Omega)$. Obviously, $Q^{\ast}\in \mathcal{P}(t,P)$. For each
$\xi^{\prime}\in L_{ip}(\Omega_{t})$, we have $E_{Q^{\ast}}[\xi^{\prime
}\varphi(\xi,\eta)]=E[\xi^{\prime}\varphi(\xi,\eta)]=E_{P}[\xi^{\prime}%
E_{Q}[\varphi(x,\eta)]_{x=\xi}]=E_{Q^{\ast}}[\xi^{\prime}E_{Q}[\varphi
(x,\eta)]_{x=\xi}]$, which implies $E_{Q^{\ast}}[\varphi(\xi,\eta
)|\mathcal{F}_{t}]=E_{Q}[\varphi(x,\eta)]_{x=\xi}$, $P$-a.s..
\end{proof}

\begin{lemma}
\label{le2.18} For each fixed $m$, $n\geq1$, $\varphi \in C_{b.Lip}%
(\mathbb{R}^{m+n})$ and $\eta \in L_{ip}(\Omega;\mathbb{R}^{n})$, we set
$\Lambda(\xi)=\underset{Q\in \mathcal{P}(t,P)}{ess\sup}^{P}E_{Q}[\varphi
(\xi,\eta)|\mathcal{F}_{t}]$ for each $\xi \in \mathbb{L}^{1}(\Omega
_{t};\mathbb{R}^{m})$. Then $\Lambda(\xi)=\Lambda(x)|_{x=\xi}$, $P$-a.s..
\end{lemma}

\begin{proof}
For each $\xi_{1}$, $\xi_{2}\in \mathbb{L}^{1}(\Omega_{t};\mathbb{R}^{m})$, we
have%
\[
|\Lambda(\xi_{1})-\Lambda(\xi_{2})|\leq \underset{Q\in \mathcal{P}(t,P)}%
{ess\sup}^{P}E_{Q}[|\varphi(\xi_{1},\eta)-\varphi(\xi_{2},\eta)||\mathcal{F}%
_{t}]\leq L|\xi_{1}-\xi_{2}|,\text{ }P\text{-a.s.},
\]
where $L$ is the Lipschitz constant of $\varphi$. From this we can deduce that
$|\Lambda(x)|_{x=\xi_{1}}-\Lambda(x)|_{x=\xi_{2}}|\leq L|\xi_{1}-\xi_{2}|,$
$P$-a.s.. Thus we only need to consider bounded $\xi \in \mathbb{L}^{1}%
(\Omega_{t};\mathbb{R}^{m})$. For each $\varepsilon>0$, we can choose a simple
function $\eta^{\varepsilon}=\sum_{i=1}^{n}x_{i}I_{A_{i}}$, where
$(A_{i})_{i=1}^{n}$ is an $\mathcal{F}_{t}$-partition of $\Omega$ and
$x_{i}\in \mathbb{R}^{m}$, such that $|\eta^{\varepsilon}-\xi|\leq \varepsilon$.
Then we get $|\Lambda(\xi)-\Lambda(\eta^{\varepsilon})|\leq L\varepsilon,$
$P$-a.s.$,$ and $|\Lambda(x)|_{x=\xi}-\Lambda(x)|_{x=\eta^{\varepsilon}}|\leq
L\varepsilon$, $P$-a.s.. On the other hand, we have%
\begin{align*}
\Lambda(\eta^{\varepsilon})  &  =\underset{Q\in \mathcal{P}(t,P)}{ess\sup}%
^{P}\sum_{i=1}^{n}I_{A_{i}}E_{Q}[\varphi(x_{i},\eta)|\mathcal{F}_{t}]\\
&  =\sum_{i=1}^{n}I_{A_{i}}\underset{Q\in \mathcal{P}(t,P)}{ess\sup}^{P}%
E_{Q}[\varphi(x_{i},\eta)|\mathcal{F}_{t}]=\Lambda(x)|_{x=\eta^{\varepsilon}%
},\text{ }P\text{-a.s..}%
\end{align*}
Thus we get $|\Lambda(\xi)-\Lambda(x)|_{x=\xi}|\leq2L\varepsilon,$ $P$-a.s..
Letting $\varepsilon \rightarrow0$, we obtain $\Lambda(\xi)=\Lambda(x)|_{x=\xi
}$, $P$-a.s..
\end{proof}

\begin{remark}
It is important to note that%
\[
I_{A_{j}}\underset{Q\in \mathcal{P}(t,P)}{ess\sup}^{P}\sum_{i=1}^{n}I_{A_{i}%
}E_{Q}[\varphi(x_{i},\eta)|\mathcal{F}_{t}]=\underset{Q\in \mathcal{P}%
(t,P)}{ess\sup}^{P}I_{A_{j}}E_{Q}[\varphi(x_{j},\eta)|\mathcal{F}_{t}],
\]
which implies%
\[
\underset{Q\in \mathcal{P}(t,P)}{ess\sup}^{P}\sum_{i=1}^{n}I_{A_{i}}%
E_{Q}[\varphi(x_{i},\eta)|\mathcal{F}_{t}]=\sum_{i=1}^{n}I_{A_{i}}%
\underset{Q\in \mathcal{P}(t,P)}{ess\sup}^{P}E_{Q}[\varphi(x_{i},\eta
)|\mathcal{F}_{t}].
\]

\end{remark}

\begin{lemma}
\label{le2.19}For each $X\in L_{ip}(\Omega)$, we have, for each $P\in
\mathcal{P}$,%
\[
\mathbb{\hat{E}}_{t}[X]=\underset{Q\in \mathcal{P}(t,P)}{ess\sup}^{P}%
E_{Q}[X|\mathcal{F}_{t}]\text{ \  \ }P\text{-a.s.}.
\]

\end{lemma}

\begin{proof}
By Lemma \ref{le2.16}, we know that $\underset{Q\in \mathcal{P}(t,P)}{ess\sup
}^{P}E_{Q}[X|\mathcal{F}_{t}]\leq \mathbb{\hat{E}}_{t}[X]$, $P$-a.s., for each
$X\in L_{ip}(\Omega)$. We only need to prove that $\underset{Q\in
\mathcal{P}(t,P)}{ess\sup}^{P}E_{Q}[X|\mathcal{F}_{t}]\geq \mathbb{\hat{E}}%
_{t}[X]$, $P$-a.s..

Step 1. For each $X\in L_{ip}(\Omega^{t})$, we can choose a $Q^{\ast}%
\in \mathcal{P}$ such that $E_{Q^{\ast}}[X]=\mathbb{\hat{E}}[X]=\mathbb{\hat
{E}}_{t}[X]$. By Lemma \ref{le2.17}, there exists a $Q\in \mathcal{P}(t,P)$
such that $E_{Q}[X|\mathcal{F}_{t}]=E_{Q^{\ast}}[X]$, $P$-a.s.. Thus
$\underset{Q\in \mathcal{P}(t,P)}{ess\sup}^{P}E_{Q}[X|\mathcal{F}_{t}%
]\geq \mathbb{\hat{E}}_{t}[X]$, $P$-a.s., for each $X\in L_{ip}(\Omega^{t})$.

Step 2. For each $X\in L_{ip}(\Omega)$, there exist $t_{1}<t_{2}<\cdots<t_{n}$
with $t=t_{k}$ and $\Phi \in C_{b.Lip}(\mathbb{R}^{d\times n})$ such that
$X=\Phi(B_{t_{1}},B_{t_{2}}-B_{t_{1}},\ldots,B_{t_{n}}-B_{t_{n-1}})$. By Lemma
\ref{le2.18}, we get%
\[
\underset{Q\in \mathcal{P}(t,P)}{ess\sup}^{P}E_{Q}[X|\mathcal{F}_{t}%
]=\Psi(B_{t_{1}},\cdots,B_{t_{k}}-B_{t_{k-1}}),\text{ }P\text{-a.s.},
\]
where $\Psi(x_{1},\cdots,x_{k})=\underset{Q\in \mathcal{P}(t,P)}{ess\sup}%
^{P}E_{Q}[\Phi(x_{1},\cdots,x_{k},B_{t_{k+1}}-B_{t_{k}},\ldots,B_{t_{n}%
}-B_{t_{n-1}})|\mathcal{F}_{t}]$. By Step 1, we have $\Psi(x_{1},\cdots
,x_{k})\geq \mathbb{\hat{E}}[\Phi(x_{1},\cdots,x_{k},B_{t_{k+1}}-B_{t_{k}%
},\ldots,B_{t_{n}}-B_{t_{n-1}})]$, $P$-a.s.. Thus $\underset{Q\in
\mathcal{P}(t,P)}{ess\sup}^{P}E_{Q}[X|\mathcal{F}_{t}]\geq \mathbb{\hat{E}}%
_{t}[X]$, $P$-a.s., for each $X\in L_{ip}(\Omega)$.
\end{proof}

We now consider the representation of the conditional $G$-expectation for
$X\in L_{G}^{1}(\Omega)$.

\begin{lemma}
\label{le2.23} Let $A_{1},\ldots,A_{n}$ be an $\mathcal{F}_{t}$-partition of
$\Omega$. Then for each given $\{Q_{i}\}_{i=1}^n\subset\mathcal{P}(t,P)$, there
exists a unique $Q\in \mathcal{P}(t,P)$ such that $E_{Q}[X|\mathcal{F}%
_{t}]=\sum_{i=1}^{n}E_{Q_{i}}[X|\mathcal{F}_{t}]I_{A_{i}}$ for each $X\in
L_{ip}(\Omega)$.
\end{lemma}

\begin{proof}
For each $X\in L_{ip}(\Omega)$, we set $E[X]=E_{P}[\sum_{i=1}^{n}E_{Q_{i}%
}[X|\mathcal{F}_{t}]I_{A_{i}}]$. It is easy to check that $E[\cdot]$ is a
linear expectation on $L_{ip}(\Omega)$. By Lemma \ref{le2.16}, we obtain that
$E[X]\leq \mathbb{\hat{E}}[X]$ for each $X\in L_{ip}(\Omega)$. By Theorem
\ref{the2.10} and the Daniell-Stone theorem, there exists a unique
$Q\in \mathcal{P}$ such that $E[X]=E_{Q}[X]$ for each $X\in L_{ip}(\Omega)$. It
is easy to verify that $Q\in \mathcal{P}(t,P)$ and $E_{Q}[X|\mathcal{F}%
_{t}]=\sum_{i=1}^{n}E_{Q_{i}}[X|\mathcal{F}_{t}]I_{A_{i}}$ for each $X\in
L_{ip}(\Omega)$.
\end{proof}

\begin{lemma}
\label{le2.24} For each fixed $t\geq0$ and $P\in \mathcal{P}$, we have

\begin{description}
\item[(1)] $\mathcal{P}(t,P)$ is weakly compact;

\item[(2)] For each $\xi \in \mathbb{L}^{1}(\Omega)$, there exists a sequence
$Q_{n}\in \mathcal{P}(t,P)$ such that
\[
E_{Q_{n}}[\xi|\mathcal{F}_{t}]\uparrow \underset{Q\in \mathcal{P}(t,P)}{ess\sup
}^{P}E_{Q}[\xi|\mathcal{F}_{t}],P-a.s.\text{;}%
\]

\item[(3)] For each $\xi \in \mathbb{L}^{1}(\Omega)$, $E_{P}[\underset
{Q\in \mathcal{P}(t,P)}{ess\sup}^{P}E_{Q}[\xi|\mathcal{F}_{t}]]=\underset
{Q\in \mathcal{P}(t,P)}{\sup}E_{Q}[\xi]$.
\end{description}
\end{lemma}

\begin{proof}
(1) is due to the definition of weak convergence and $L_{ip}(\Omega)\in
C_{b}(\Omega)$. For each $Q$, $Q^{\prime}\in \mathcal{P}(t,P)$, by Lemma
\ref{le2.23}, it is easy to deduce that there exists a $Q^{\ast}\in
\mathcal{P}(t,P)$ such that $E_{Q^{\ast}}[\xi|\mathcal{F}_{t}]=E_{Q}%
[\xi|\mathcal{F}_{t}]\vee E_{Q^{\prime}}[\xi|\mathcal{F}_{t}]$, which implies
(2). (3) can be easily obtained by (2).
\end{proof}

\textbf{Proof of Theorem \ref{the2.15}.} For each $X\in L_{G}^{1}(\Omega)$,
there exists a sequence $X_{n}\in L_{ip}(\Omega)$ such that $X_{n}%
\overset{\mathbb{L}^{1}}{\rightarrow}X$. It is easy to verify that
$\mathbb{\hat{E}}_{t}[X_{n}]\overset{\mathbb{L}^{1}}{\rightarrow}%
\mathbb{\hat{E}}_{t}[X]$. By Theorem \ref{the2.12} and Lemma \ref{le2.19}, we
can get $\mathbb{\hat{E}}_{t}[X]\geq \underset{Q\in \mathcal{P}(t,P)}{ess\sup
}^{P}E_{Q}[X|\mathcal{F}_{t}]$ \  \ $P$-a.s.. Now we assert $P(\mathbb{\hat{E}%
}_{t}[X]>\underset{Q\in \mathcal{P}(t,P)}{ess\sup}^{P}E_{Q}[X|\mathcal{F}%
_{t}])=0$. Otherwise, by Lemma \ref{le2.24}, we obtain
\[
E_{P}[\mathbb{\hat{E}}_{t}[X]]>E_{P}[\underset{Q\in \mathcal{P}(t,P)}{ess\sup
}^{P}E_{Q}[X|\mathcal{F}_{t}]]=\underset{Q\in \mathcal{P}(t,P)}{\sup}E_{Q}[X].
\]
On the other hand, by Lemma \ref{le2.19}, we get%
\[
E_{P}[\mathbb{\hat{E}}_{t}[X]]=\lim_{n\rightarrow \infty}E_{P}[\mathbb{\hat{E}%
}_{t}[X_{n}]]=\lim_{n\rightarrow \infty}\underset{Q\in \mathcal{P}(t,P)}{\sup
}E_{Q}[X_{n}].
\]
Noting that $|\underset{Q\in \mathcal{P}(t,P)}{\sup}E_{Q}[X_{n}]-\underset
{Q\in \mathcal{P}(t,P)}{\sup}E_{Q}[X]|\leq \mathbb{\hat{E}}[|X_{n}-X|]$, then we
get a contradiction, which implies the result.

\section{Extension of conditional $G$-expectations}

\subsection{Extension from above}

We set%
\[
L_{G}^{1^{\ast}}(\Omega)=\{X\in \mathbb{L}^{1}(\Omega):\exists X_{n}\in
L_{G}^{1}(\Omega)\text{ such that }X_{n}\downarrow X\text{ q.s.}\},
\]%
\[
\mathcal{L}_{G}^{1^{\ast}}(\Omega)=\{X\in L^{0}(\Omega):\exists X_{n}\in
L_{G}^{1}(\Omega)\text{ such that }X_{n}\downarrow X\text{ q.s.}\}.
\]
Similarly,
\[
L_{G}^{1^{\ast}}(\Omega_{t})=\{X\in \mathbb{L}^{1}(\Omega_{t}):\exists X_{n}\in
L_{G}^{1}(\Omega_{t})\text{ such that }X_{n}\downarrow X\text{ q.s.}\},
\]%
\[
\mathcal{L}_{G}^{1^{\ast}}(\Omega_{t})=\{X\in L^{0}(\Omega_{t}):\exists
X_{n}\in L_{G}^{1}(\Omega_{t})\text{ such that }X_{n}\downarrow X\text{
q.s.}\}.
\]

Obviously, $L_{G}^{1^{\ast}}(\Omega)\subset \mathcal{L}_{G}^{1^{\ast}}(\Omega
)$. Now we give the extension of the conditional $G$-expectation from above.

\begin{definition}
\label{de1}For each $X\in \mathcal{L}_{G}^{1^{\ast}}(\Omega)$, there exists a
sequence $\{X_{n}\}_{n=1}^{\infty}\subset L_{G}^{1}(\Omega)$ such that
$X_{n}\downarrow X$ q.s., we define%
\[
\mathbb{\hat{E}}_{t}[X]=\lim_{n\rightarrow \infty}\mathbb{\hat{E}}_{t}%
[X_{n}]\text{ q.s..}%
\]

\end{definition}

We first prove that the above definition does not depend on a particular
sequence $\{X_{n}\}_{n=1}^{\infty}$.

\begin{lemma}
\label{le1}Let $\xi_{n}\in L_{G}^{1}(\Omega)$, $\xi_{n}\downarrow0$ q.s.. Then
$\mathbb{\hat{E}}_{t}[\xi_{n}]\downarrow0$ q.s..
\end{lemma}

\begin{proof}
By Theorem \ref{the2.10}, we have $\mathbb{\hat{E}}[\xi_{n}]\downarrow0$. Let
$\eta=\lim_{n\rightarrow \infty}$ $\mathbb{\hat{E}}_{t}[\xi_{n}]$, then
$\eta \geq0$ q.s. and $\mathbb{\hat{E}}[\eta]\leq \mathbb{\hat{E}}[\xi_{n}]$ for
each $n$. Thus $\mathbb{\hat{E}}[\eta]=0$, which implies $\eta=0$ q.s..
\end{proof}

\begin{proposition}
\label{pro2}Let $X\in \mathcal{L}_{G}^{1^{\ast}}(\Omega)$ and let
$\{X_{n}\}_{n=1}^{\infty}$, $\{ \tilde{X}_{n}\}_{n=1}^{\infty}$ be two
sequences in $L_{G}^{1}(\Omega)$ such that $X_{n}\downarrow X$ and $\tilde
{X}_{n}\downarrow X$ q.s.. Then
\[
\lim_{n\rightarrow \infty}\mathbb{\hat{E}}_{t}[X_{n}]=\lim_{n\rightarrow \infty
}\mathbb{\hat{E}}_{t}[\tilde{X}_{n}],\  \  \text{q.s..}%
\]

\end{proposition}

\begin{proof}
For each fixed $m$,
\[
\mathbb{\hat{E}}_{t}[X_{n}]=\mathbb{\hat{E}}_{t}[\tilde{X}_{m}+X_{n}-\tilde
{X}_{m}]\leq \mathbb{\hat{E}}_{t}[\tilde{X}_{m}]+\mathbb{\hat{E}}_{t}%
[X_{n}-\tilde{X}_{m}]\leq \mathbb{\hat{E}}_{t}[\tilde{X}_{m}]+\mathbb{\hat{E}%
}_{t}[(X_{n}-\tilde{X}_{m})^{+}].
\]
Since $X_{n}\downarrow X$ q.s., we get $(X_{n}-\tilde{X}_{m})^{+}%
\downarrow(X-\tilde{X}_{m})^{+}=0$ q.s.. By Lemma \ref{le1}, we obtain
$\lim_{n\rightarrow \infty}\mathbb{\hat{E}}_{t}[X_{n}]$ $\leq \mathbb{\hat{E}%
}_{t}[\tilde{X}_{m}]$ q.s.. Thus $\lim_{n\rightarrow \infty}\mathbb{\hat{E}%
}_{t}[X_{n}]$ $\leq \lim_{n\rightarrow \infty}\mathbb{\hat{E}}_{t}[\tilde{X}%
_{n}]$ q.s.. Similarly, we can prove that $\lim_{n\rightarrow \infty
}\mathbb{\hat{E}}_{t}[\tilde{X}_{n}]\leq \lim_{n\rightarrow \infty}%
\mathbb{\hat{E}}_{t}[X_{n}]$ q.s.. The proof is complete.
\end{proof}

\begin{remark}
For each given $X\in \mathcal{L}_{G}^{1^{\ast}}(\Omega)$, there exists a
sequence $\{X_{n}\}_{n=1}^{\infty}\subset L_{G}^{1}(\Omega)$ such that
$X_{n}\downarrow X$ q.s.. We have $\mathbb{\hat{E}}_{t}[X_{n}]\downarrow
\mathbb{\hat{E}}_{t}[X]$ q.s., but we do not have $\mathbb{\hat{E}%
}[|\mathbb{\hat{E}}_{t}[X_{n}]-\mathbb{\hat{E}}_{t}[X]|]\rightarrow0$.
\end{remark}

Now we study the properties of the above conditional $G$-expectation.

\begin{proposition}
\label{pro3}We have

\begin{description}
\item[(1)] $\mathbb{\hat{E}}_{0}[X]=\mathbb{\hat{E}}[X]$ for $X\in
\mathcal{L}_{G}^{1^{\ast}}(\Omega)$;

\item[(2)] $X,Y\in \mathcal{L}_{G}^{1^{\ast}}(\Omega)$, $X\leq Y\ $%
q.s.$\Longrightarrow \mathbb{\hat{E}}_{t}[X]\leq \mathbb{\hat{E}}_{t}%
[Y]\  \ $q$_{.}$s$_{.}$;

\item[(3)] $X\in \mathcal{L}_{G}^{1^{\ast}}(\Omega_{t})$, $Y\in \mathcal{L}%
_{G}^{1^{\ast}}(\Omega)\Longrightarrow \mathbb{\hat{E}}_{t}[X+Y]=X+\mathbb{\hat
{E}}_{t}[Y]$;

\item[(4)] $X,Y\in \mathcal{L}_{G}^{1^{\ast}}(\Omega)\Longrightarrow
\mathbb{\hat{E}}_{t}[X+Y]\leq \mathbb{\hat{E}}_{t}[X]+\mathbb{\hat{E}}_{t}[Y]$;

\item[(5)] $X\in \mathcal{L}_{G}^{1^{\ast}}(\Omega_{t})$ is bounded, $X\geq0$,
$Y\in \mathcal{L}_{G}^{1^{\ast}}(\Omega)$, $Y\geq0\Longrightarrow
XY\in \mathcal{L}_{G}^{1^{\ast}}(\Omega)$ and $\mathbb{\hat{E}}_{t}%
[XY]=X\mathbb{\hat{E}}_{t}[Y]$;

\item[(6)] $X\in \mathcal{L}_{G}^{1^{\ast}}(\Omega)\Longrightarrow
\mathbb{\hat{E}}_{t}[X]\in \mathcal{L}_{G}^{1^{\ast}}(\Omega_{t})$ and
$\mathbb{\hat{E}}_{s}[\mathbb{\hat{E}}_{t}[X]]=\mathbb{\hat{E}}_{s\wedge
t}[X]$;

\item[(7)] $\{X_{n}\}_{n=1}^{\infty}\subset \mathcal{L}_{G}^{1^{\ast}}(\Omega
)$, $X_{n}\downarrow X$ q.s.$\Longrightarrow X\in \mathcal{L}_{G}^{1^{\ast}%
}(\Omega)$ and $\mathbb{\hat{E}}_{t}[X]=\lim_{n\rightarrow \infty}%
\mathbb{\hat{E}}_{t}[X_{n}]$ q.s.;

\item[(8)] $\mathbb{\hat{E}}_{t}[X^{+}]\geq(\mathbb{\hat{E}}_{t}[X])^{+}$ for
$X\in \mathcal{L}_{G}^{1^{\ast}}(\Omega)$;

\item[(9)] If $\{X_{n}\}_{n=1}^{\infty}\subset L_{G}^{1}(\Omega)$ such that
$X_{n}\downarrow X$ q.s. and $X\in L^{0}(\Omega_{t})$, then $\mathbb{\hat{E}%
}_{t}[X_{n}]\downarrow X$ q.s.; ;

\item[(10)] $\mathcal{L}_{G}^{1^{\ast}}(\Omega_{t})=\mathcal{L}_{G}^{1^{\ast}%
}(\Omega)\cap L^{0}(\Omega_{t})$.
\end{description}
\end{proposition}

\begin{proof}
(1) This is part (c) in Theorem \ref{the2.10}.

(2) Let $\{X_{n}\}_{n=1}^{\infty}$, $\{Y_{n}\}_{n=1}^{\infty}\subset L_{G}%
^{1}(\Omega)$ such that $X_{n}\downarrow X$ q.s. and $Y_{n}\downarrow Y$ q.s..
We have
\[
X_{n}\leq(X_{n}\vee Y_{n})\downarrow(X\vee Y)=Y,\  \  \text{q.s..}%
\]
Thus%
\[
\mathbb{\hat{E}}_{t}[Y]=\lim_{n\rightarrow \infty}\mathbb{\hat{E}}_{t}%
[X_{n}\vee Y_{n}]\geq \lim_{n\rightarrow \infty}\mathbb{\hat{E}}_{t}%
[X_{n}]=\mathbb{\hat{E}}_{t}[X],\  \  \text{ q.s..}%
\]

(3), (4) and (5) are obvious.

(6) Let $\{X_{n}\}_{n=1}^{\infty}\subset L_{G}^{1}(\Omega)$ be such that
$X_{n}\downarrow X$ q.s.. Then $\mathbb{\hat{E}}_{t}[X_{n}]\in L_{G}%
^{1}(\Omega_{t})$ and $\mathbb{\hat{E}}_{t}[X_{n}]\downarrow \mathbb{\hat{E}%
}_{t}[X]$ q.s.. Thus $\mathbb{\hat{E}}_{t}[X]\in \mathcal{L}_{G}^{1^{\ast}%
}(\Omega_{t})$. Moreover, the definition of $\mathbb{\hat{E}}_{s}$ implies
\[
\mathbb{\hat{E}}_{s}[\mathbb{\hat{E}}_{t}[X]]=\lim_{n\rightarrow \infty
}\mathbb{\hat{E}}_{s}[\mathbb{\hat{E}}_{t}[X_{n}]]=\lim_{n\rightarrow \infty
}\mathbb{\hat{E}}_{s\wedge t}[X_{n}]=\mathbb{\hat{E}}_{s\wedge t}[X].
\]

(7) Let $\{ \xi_{m}^{n}\}_{m=1}^{\infty}\subset L_{G}^{1}(\Omega)$ be such
that $\xi_{m}^{n}\downarrow X_{n}$ q.s. as $m\rightarrow \infty$. We set
$\eta_{n}=\wedge_{i=1}^{n}\xi_{n}^{i}\in L_{G}^{1}(\Omega)$. It is easy to
check that $\eta_{n}\geq X_{n}$ q.s. and $\eta_{n}\downarrow X$ q.s.. Thus we
get $X\in \mathcal{L}_{G}^{1^{\ast}}(\Omega)$ and
\[
\mathbb{\hat{E}}_{t}[X]=\lim_{n\rightarrow \infty}\mathbb{\hat{E}}_{t}[\eta
_{n}]\geq \lim_{n\rightarrow \infty}\mathbb{\hat{E}}_{t}[X_{n}],\  \text{ q.s..}%
\]
By (2) we have $\mathbb{\hat{E}}_{t}[X]\leq \lim_{n\rightarrow \infty
}\mathbb{\hat{E}}_{t}[X_{n}]$ q.s.. Thus we get $\mathbb{\hat{E}}_{t}%
[X]=\lim_{n\rightarrow \infty}\mathbb{\hat{E}}_{t}[X_{n}]$ q.s..

(8) Obviously $X^{+}\in \mathcal{L}_{G}^{1^{\ast}}(\Omega)$ and $X^{+}\geq X$.
By (2) we have $\mathbb{\hat{E}}_{t}[X^{+}]\geq \mathbb{\hat{E}}_{t}[X]$, q.s..
Thus $\mathbb{\hat{E}}_{t}[X^{+}]=(\mathbb{\hat{E}}_{t}[X^{+}])^{+}%
\geq(\mathbb{\hat{E}}_{t}[X])^{+}$.

(9) For each fixed $n$, we first prove that $\mathbb{\hat{E}}_{t}[X_{n}]\geq
X$ q.s.. Otherwise\ we have $c(\{X>\mathbb{\hat{E}}_{t}[X_{n}]\})>0$. Since
\[
\{X>\mathbb{\hat{E}}_{t}[X_{n}]+\frac{1}{k}\} \uparrow \{X>\mathbb{\hat{E}}%
_{t}[X_{n}]\} \  \text{as }k\rightarrow \infty \text{,}\
\]
by (3) and (5) in Proposition \ref{pro2.8}, we can choose a constant
$\delta>0$ and a compact set $K\in \mathcal{F}_{t}$ such that $K\subset
\{X>\mathbb{\hat{E}}_{t}[X_{n}]+\delta \}$ and $c(K)>0$. It is easy to verify
that $I_{K}\in \mathcal{L}_{G}^{1^{\ast}}(\Omega_{t})$. By (1), (5), (6) and
(8), we have%
\begin{align*}
\mathbb{\hat{E}}[(X-X_{n})^{+}I_{K}]  &  \geq \mathbb{\hat{E}}[(\mathbb{\hat
{E}}_{t}[X_{n}]+\delta-X_{n})^{+}I_{K}]\\
&  =\mathbb{\hat{E}}[\mathbb{\hat{E}}_{t}[(\mathbb{\hat{E}}_{t}[X_{n}%
]+\delta-X_{n})^{+}I_{K}]]\\
&  =\mathbb{\hat{E}}[\mathbb{\hat{E}}_{t}[(\mathbb{\hat{E}}_{t}[X_{n}%
]+\delta-X_{n})^{+}]I_{K}]\\
&  \geq \mathbb{\hat{E}}[(\mathbb{\hat{E}}_{t}[X_{n}]+\delta+\mathbb{\hat{E}%
}_{t}[-X_{n}])^{+}I_{K}]\\
&  \geq \mathbb{\hat{E}}[\delta I_{K}]=\delta c(K)>0,
\end{align*}
which contradicts to $X\leq X_{n}$ q.s.. Thus $\mathbb{\hat{E}}_{t}[X_{n}]\geq
X$ q.s. for each $n\geq1$.

Let $\eta=\lim_{n\rightarrow \infty}\mathbb{\hat{E}}_{t}[X_{n}]\geq X$ q.s..
Now we show that $\eta=X$ q.s.. Otherwise $c(\{ \eta>X\})>0$. Obviously
$c(\{X=\infty \})=0$. We divide the proof into two cases.

Case 1: $c(\{ \eta>X\} \cap \{X=-\infty \})>0$. \ Since $\{ \eta>-M\}
\cap \{X=-\infty \} \uparrow \{ \eta>X\} \cap \{X=-\infty \}$ as $M\rightarrow
\infty$, by (3) and (5) in Proposition \ref{pro2.8}, we can choose a constant
$M>0$ and a compact set $K\in \mathcal{F}_{t}$ such that $K\subset \{ \eta>-M\}
\cap \{X=-\infty \}$ and $c(K)>0$. By Theorem \ref{the2.11}, we can get%
\begin{align*}
\mathbb{\hat{E}}[(X_{n}+M+1)I_{K}]  &  =\mathbb{\hat{E}}[(X_{n}+M+1)I_{K}%
+0I_{K^{c}}]\\
&  =\mathbb{\hat{E}}[I_{K}(\mathbb{\hat{E}}_{t}[X_{n}]+M+1)]\\
&  \geq \mathbb{\hat{E}}[I_{K}(\eta+M+1)]\\
&  \geq \mathbb{\hat{E}}[I_{K}]=c(K)>0.
\end{align*}
On the other hand, it is easy to check that $(X_{n}+M+1)^{+}I_{K}%
\in \mathcal{L}_{G}^{1^{\ast}}(\Omega)$ and $(X_{n}+M+1)^{+}I_{K}%
\downarrow(X+M+1)^{+}I_{K}=0$ q.s.. By (1) and (7) we get $\mathbb{\hat{E}%
}[(X_{n}+M+1)^{+}I_{K}]\downarrow0$. But $\mathbb{\hat{E}}[(X_{n}%
+M+1)I_{K}]\leq \mathbb{\hat{E}}[(X_{n}+M+1)^{+}I_{K}]$, which implies a
contradiction as $n\rightarrow \infty$.

Case 2: $c(\{ \eta>X\} \cap \{|X|<\infty \})>0$. Since $\{ \eta>X+\frac{1}{k}\}
\cap \{|X|<k\} \uparrow \{ \eta>X\} \cap \{|X|<\infty \}$ as $k\rightarrow \infty$,
by (3) and (5) in Proposition \ref{pro2.8}, we can choose a constant $k>0$ and
a compact set $K\in \mathcal{F}_{t}$ such that $K\subset \{ \eta>X+\frac{1}{k}\}
\cap \{|X|<k\}$ and $c(K)>0$. It is easy to check that $(X_{n}+k)I_{K}%
=(X_{n}+k)^{+}I_{K}\in \mathcal{L}_{G}^{1^{\ast}}(\Omega)\downarrow
(X+k)I_{K}\geq0$ q.s.. By (1), (3) and (7), we obtain $\mathbb{\hat{E}}%
[(X_{n}+k)I_{K}]\downarrow \mathbb{\hat{E}}[(X+k)I_{K}]$ and%
\begin{align*}
\mathbb{\hat{E}}[(X_{n}+k)I_{K}]  &  =\mathbb{\hat{E}}[\mathbb{\hat{E}}%
_{t}[(X_{n}+k)^{+}I_{K}]]=\mathbb{\hat{E}}[\mathbb{\hat{E}}_{t}[(X_{n}%
+k)^{+}]I_{K}]\\
&  \geq \mathbb{\hat{E}}[(\mathbb{\hat{E}}_{t}[X_{n}]+k)^{+}I_{K}%
]\geq \mathbb{\hat{E}}[(\eta+k)^{+}I_{K}]\\
&  \geq \mathbb{\hat{E}}[(X+\frac{1}{k}+k)^{+}I_{K}]=\mathbb{\hat{E}}%
[(X+\frac{1}{k}+k)I_{K}],
\end{align*}
which implies $\mathbb{\hat{E}}[(X+k)I_{K}]\geq \mathbb{\hat{E}}[(X+\frac{1}%
{k}+k)I_{K}]$. This is a contradiction by the following Proposition.

(10) can be easily deduced from (9).
\end{proof}

The following Proposition is useful in this paper.

\begin{proposition}
\label{pro4}Let $A\in \mathcal{B}(\Omega)$ with $c(A)>0$, $X\in \mathbb{L}%
^{1}(\Omega)$ be such that $0\leq XI_{A}\leq M$, where $M>0$ is a constant.
Then for each $\delta>0$, we have $\mathbb{\hat{E}}[(X+\delta)I_{A}%
]>\mathbb{\hat{E}}[XI_{A}]$.
\end{proposition}

\begin{proof}
Otherwise, $\mathbb{\hat{E}}[XI_{A}]=\mathbb{\hat{E}}[(X+\delta)I_{A}%
]\geq \mathbb{\hat{E}}[\delta I_{A}]=\delta c(A)>0$. By the definition of
$\mathbb{\hat{E}}[XI_{A}]$, there exists a sequence $P_{n}\in \mathcal{P}$ such
that $E_{P_{n}}[XI_{A}]\geq \mathbb{\hat{E}}[XI_{A}]-\frac{1}{n}\delta
c(A)\geq \frac{n-1}{n}\delta c(A)$. On the other hand, $E_{P_{n}}[XI_{A}]\leq
MP_{n}(A)$, which implies $P_{n}(A)\geq \frac{1}{2M}\delta c(A)$ for each
$n\geq2$. Thus we get%
\begin{align*}
\mathbb{\hat{E}}[(X+\delta)I_{A}]  &  \geq \sup_{n}E_{P_{n}}[(X+\delta
)I_{A}]=\sup_{n}\{E_{P_{n}}[XI_{A}]+\delta P_{n}(A)\} \\
&  \geq \sup_{n}\{E_{P_{n}}[XI_{A}]+\frac{1}{2M}\delta^{2}c(A)\}=\frac{1}%
{2M}\delta^{2}c(A)+\sup_{n}E_{P_{n}}[XI_{A}]\\
&  =\frac{1}{2M}\delta^{2}c(A)+\mathbb{\hat{E}}[XI_{A}]>\mathbb{\hat{E}%
}[XI_{A}],
\end{align*}
which implies a contradiction, the proof is complete.
\end{proof}

\subsection{Extension from below}

We set%
\[
L_{G}^{1_{\ast}^{\ast}}(\Omega)=\{X\in \mathbb{L}^{1}(\Omega):\exists X_{n}\in
L_{G}^{1^{\ast}}(\Omega)\text{ such that }X_{n}\uparrow X\text{ q.s.}\},
\]%
\[
\mathcal{L}_{G}^{1_{\ast}^{\ast}}(\Omega)=\{X\in L^{0}(\Omega):\exists
X_{n}\in L_{G}^{1^{\ast}}(\Omega)\text{ such that }X_{n}\uparrow X\text{
q.s.}\}.
\]
Similarly,
\[
L_{G}^{1_{\ast}^{\ast}}(\Omega_{t})=\{X\in \mathbb{L}^{1}(\Omega_{t}):\exists
X_{n}\in L_{G}^{1^{\ast}}(\Omega_{t})\text{ such that }X_{n}\uparrow X\text{
q.s.}\},
\]%
\[
\mathcal{L}_{G}^{1_{\ast}^{\ast}}(\Omega_{t})=\{X\in L^{0}(\Omega_{t}):\exists
X_{n}\in L_{G}^{1^{\ast}}(\Omega_{t})\text{ such that }X_{n}\uparrow X\text{
q.s.}\}.
\]
Obviously, $L_{G}^{1^{\ast}}(\Omega)\subset L_{G}^{1_{\ast}^{\ast}}%
(\Omega)\subset \mathcal{L}_{G}^{1_{\ast}^{\ast}}(\Omega)$. Now we give the
extension of the conditional $G$-expectation from below.

\begin{definition}
\label{de2}For each $X\in \mathcal{L}_{G}^{1_{\ast}^{\ast}}(\Omega)$, there
exists a sequence $\{X_{n}\}_{n=1}^{\infty}\subset L_{G}^{1^{\ast}}(\Omega)$
such that $X_{n}\uparrow X$ q.s., we define%
\[
\mathbb{\hat{E}}_{t}[X]=\lim_{n\rightarrow \infty}\mathbb{\hat{E}}_{t}%
[X_{n}]\text{ q.s..}%
\]

\end{definition}

We first prove that the above definition of $\mathbb{\hat{E}}_{t}[X]$ does not
depend on a particular choice of $X_{n}\uparrow X$:

\begin{lemma}
\label{le3.1} Let $X\in \mathcal{L}_{G}^{1^{\ast}}(\Omega)$ and a closed set
$K\in \mathcal{F}_{t}$ with $X\geq0$ on $K$. Then $\mathbb{\hat{E}}_{t}%
[X^{+}]I_{K}=\mathbb{\hat{E}}_{t}[X]I_{K}$ q.s..
\end{lemma}

\begin{proof}
There exists a sequence $\{X_{n}\}_{n=1}^{\infty}\subset L_{G}^{1}(\Omega)$
such that $X_{n}\downarrow X$. Noting that $X_{n}^{+}I_{K}=X_{n}I_{K}$, then,
by Theorem \ref{the2.11}, we get $\mathbb{\hat{E}}_{t}[X_{n}^{+}%
]I_{K}=\mathbb{\hat{E}}_{t}[X_{n}]I_{K}$ q.s.. By Definition \ref{de1}, we get
$\mathbb{\hat{E}}_{t}[X^{+}]I_{K}=\mathbb{\hat{E}}_{t}[X]I_{K}$ q.s. by taking
$n\rightarrow \infty$.
\end{proof}

\begin{lemma}
\label{lem6}Let $X\in L_{G}^{1^{\ast}}(\Omega)$, $\{X_{n}\}_{n=1}^{\infty
}\subset L_{G}^{1^{\ast}}(\Omega)$ be such that $X_{n}\uparrow X$ q.s.. Then
$\mathbb{\hat{E}}_{t}[X]=\lim_{n\rightarrow \infty}\mathbb{\hat{E}}_{t}[X_{n}]$ q.s..
\end{lemma}

\begin{proof}
From monotonicity of $\mathbb{\hat{E}}_{t}$, we only need to prove that
$\eta:=\lim_{n\rightarrow \infty}\mathbb{\hat{E}}_{t}[X_{n}]\geq \mathbb{\hat
{E}}_{t}[X]$ q.s.. Otherwise, $c(\{ \eta<\mathbb{\hat{E}}_{t}[X]\})>0$. Since
$|\mathbb{\hat{E}}_{t}[X]|+|\mathbb{\hat{E}}_{t}[X_{1}]|<\infty$ q.s., we have%
\[
\{ \eta+\frac{1}{k}<\mathbb{\hat{E}}_{t}[X]\} \cap \{|\eta|\leq k\} \cap
\{X_{1}\geq-k\} \uparrow \{ \eta<\mathbb{\hat{E}}_{t}[X]\} \text{ q.s..}%
\]
Thus by (3) and (5) in Proposition \ref{pro2.8}, we can choose a constant
$k>0$ and a compact set $K\in \mathcal{F}_{t}$ such that $K\subset \{ \eta
+\frac{1}{k}<\mathbb{\hat{E}}_{t}[X]\} \cap \{|\eta|\leq k\} \cap \{X_{1}%
\geq-k\}$ and $c(K)>0$. Noting that $(X_{n}+k)I_{K}\uparrow(X+k)I_{K}$, then,
by Lemma \ref{le3.1} and (3), (5) and (6) in Proposition \ref{pro3}, we get%
\begin{align*}
\mathbb{\hat{E}}[(X+k)I_{K}]  &  =\lim_{n\rightarrow \infty}\mathbb{\hat{E}%
}[(X_{n}+k)I_{K}]=\lim_{n\rightarrow \infty}\mathbb{\hat{E}}[\mathbb{\hat{E}%
}_{t}[(X_{n}+k)^{+}I_{K}]]\\
&  =\lim_{n\rightarrow \infty}\mathbb{\hat{E}}[\mathbb{\hat{E}}_{t}%
[(X_{n}+k)^{+}]I_{K}]=\lim_{n\rightarrow \infty}\mathbb{\hat{E}}[\mathbb{\hat
{E}}_{t}[X_{n}+k]I_{K}]\\
&  =\lim_{n\rightarrow \infty}\mathbb{\hat{E}}[(\mathbb{\hat{E}}_{t}%
[X_{n}]+k)I_{K}]=\mathbb{\hat{E}}[(\mathbb{\eta}+k)I_{K}].
\end{align*}

On the other hand, by $\eta \leq X$ we get $\mathbb{\hat{E}}[(X+k)I_{K}%
]=\mathbb{\hat{E}}[(\mathbb{\hat{E}}_{t}[X]+k)I_{K}]\geq \mathbb{\hat{E}}%
[(\eta+\frac{1}{k}+k)I_{K}]$. Thus we get $\mathbb{\hat{E}}[(\mathbb{\eta
}+k)I_{K}]\geq \mathbb{\hat{E}}[(\eta+\frac{1}{k}+k)I_{K}]$, which induces a
contradiction by Proposition \ref{pro4}. The proof is complete.
\end{proof}

\begin{proposition}
\label{pro6} For a given $X\in \mathcal{L}_{G}^{1_{\ast}^{\ast}}(\Omega)$, let
$\{X_{n}\}_{n=1}^{\infty}$ and $\{ \tilde{X}_{n}\}_{n=1}^{\infty}$ be two
sequence in $L_{G}^{1^{\ast}}(\Omega)$ such that $X_{n}\uparrow X$ and
$\tilde{X}_{n}\uparrow X$ q.s.. Then $\lim_{n\rightarrow \infty}\mathbb{\hat
{E}}_{t}[X_{n}]=\lim_{n\rightarrow \infty}\mathbb{\hat{E}}_{t}[\tilde{X}_{n}]$ q.s..
\end{proposition}

\begin{proof}
\ Indeed, for each fixed $m$, we have $X_{n}\wedge \tilde{X}_{m}\uparrow
\tilde{X}_{m}$ q.s. as $n\rightarrow \infty$. It follows from Lemma \ref{lem6}
that
\[
\lim_{n\rightarrow \infty}\mathbb{\hat{E}}_{t}[X_{n}]\geq \lim_{n\rightarrow
\infty}\mathbb{\hat{E}}_{t}[X_{n}\wedge \tilde{X}_{m}]=\mathbb{\hat{E}}%
_{t}[\tilde{X}_{m}]\text{ q.s..}%
\]
Thus $\lim_{n\rightarrow \infty}\mathbb{\hat{E}}_{t}[X_{n}]\geq \lim
_{n\rightarrow \infty}\mathbb{\hat{E}}_{t}[\tilde{X}_{n}]$. Similarly we have
$\lim_{n\rightarrow \infty}\mathbb{\hat{E}}_{t}[\tilde{X}_{n}]\geq
\lim_{n\rightarrow \infty}\mathbb{\hat{E}}_{t}[X_{n}]$.
\end{proof}

We now give the properties of the conditional $G$-expectation.

\begin{proposition}
\label{pro7} We have

\begin{description}
\item[(1)] $\mathbb{\hat{E}}_{0}[X]=\mathbb{\hat{E}}[X]$ for $X\in
\mathcal{L}_{G}^{1_{\ast}^{\ast}}(\Omega)$;

\item[(2)] $X,Y\in \mathcal{L}_{G}^{1_{\ast}^{\ast}}(\Omega)$, $X\leq
Y\Longrightarrow \mathbb{\hat{E}}_{t}[X]\leq \mathbb{\hat{E}}_{t}[Y]$;

\item[(3)] $X\in \mathcal{L}_{G}^{1_{\ast}^{\ast}}(\Omega_{t})$, $Y\in
\mathcal{L}_{G}^{1_{\ast}^{\ast}}(\Omega)\Longrightarrow \mathbb{\hat{E}}%
_{t}[X+Y]=X+\mathbb{\hat{E}}_{t}[Y]$;

\item[(4)] $X,Y\in \mathcal{L}_{G}^{1_{\ast}^{\ast}}(\Omega)\Longrightarrow
\mathbb{\hat{E}}_{t}[X+Y]\leq \mathbb{\hat{E}}_{t}[X]+\mathbb{\hat{E}}_{t}[Y]$;

\item[(5)] $X\in \mathcal{L}_{G}^{1_{\ast}^{\ast}}(\Omega_{t})$, $X\geq0$,
$Y\in \mathcal{L}_{G}^{1_{\ast}^{\ast}}(\Omega)$, $Y\geq0\Longrightarrow
XY\in \mathcal{L}_{G}^{1_{\ast}^{\ast}}(\Omega)$ and $\mathbb{\hat{E}}%
_{t}[XY]=X\mathbb{\hat{E}}_{t}[Y]$;

\item[(6)] $X\in \mathcal{L}_{G}^{1_{\ast}^{\ast}}(\Omega)\Longrightarrow
\mathbb{\hat{E}}_{t}[X]\in \mathcal{L}_{G}^{1_{\ast}^{\ast}}(\Omega_{t})$ and
$\mathbb{\hat{E}}_{s}[\mathbb{\hat{E}}_{t}[X]]=\mathbb{\hat{E}}_{s\wedge
t}[X]$;

\item[(7)] $X_{n}\in \mathcal{L}_{G}^{1_{\ast}^{\ast}}(\Omega)$, $X_{n}\uparrow
X\Longrightarrow X\in \mathcal{L}_{G}^{1_{\ast}^{\ast}}(\Omega)$ and
$\mathbb{\hat{E}}_{t}[X]=\lim_{n\rightarrow \infty}\mathbb{\hat{E}}_{t}[X_{n}]$ q.s.;

\item[(8)] $\mathbb{\hat{E}}_{t}[X^{+}]\geq(\mathbb{\hat{E}}_{t}[X])^{+}$ for
$X\in \mathcal{L}_{G}^{1_{\ast}^{\ast}}(\Omega)$;

\item[(9)] If $\{X_{n}\}_{n=1}^{\infty}\subset L_{G}^{1^{\ast}}(\Omega)$ such
that $X_{n}\uparrow X$ q.s. and $X\in L^{0}(\Omega_{t})$, then $\mathbb{\hat
{E}}_{t}[X_{n}]\uparrow X$ q.s.;

\item[(10)] $\mathcal{L}_{G}^{1_{\ast}^{\ast}}(\Omega_{t})=\mathcal{L}%
_{G}^{1_{\ast}^{\ast}}(\Omega)\cap L^{0}(\Omega_{t})$.
\end{description}
\end{proposition}

\begin{proof}
Similar to the proof of Proposition \ref{pro3}, we can get (1)-(8) and (10).
We only prove (9). For each fixed $n$, we first prove that $\mathbb{\hat{E}%
}_{t}[X_{n}]\leq X$ q.s.. Otherwise\ we have $c(\{X<\mathbb{\hat{E}}_{t}%
[X_{n}]\})>0$. Since
\[
\{X+\frac{1}{k}<\mathbb{\hat{E}}_{t}[X_{n}]\} \uparrow \{X<\mathbb{\hat{E}}%
_{t}[X_{n}]\} \  \text{as }k\rightarrow \infty \text{,}\
\]
by (3) and (5) in Proposition \ref{pro2.8}, we can choose a constant
$\delta>0$ and a compact set $K\in \mathcal{F}_{t}$ such that $K\subset
\{X+\delta<\mathbb{\hat{E}}_{t}[X_{n}]\}$ and $c(K)>0$. It is easy to verify
that $I_{K}\in \mathcal{L}_{G}^{1^{\ast}}(\Omega_{t})$ and $-\mathbb{\hat{E}%
}_{t}[X_{n}]\in \mathcal{L}_{G}^{1_{\ast}^{\ast}}(\Omega)$. By (1), (3), (5),
(6) and (8), we have%
\begin{align*}
\mathbb{\hat{E}}[(X_{n}-X)^{+}I_{K}]  &  \geq \mathbb{\hat{E}}[(X_{n}%
+\delta-\mathbb{\hat{E}}_{t}[X_{n}])^{+}I_{K}]\\
&  =\mathbb{\hat{E}}[\mathbb{\hat{E}}_{t}[(X_{n}+\delta-\mathbb{\hat{E}}%
_{t}[X_{n}])^{+}]I_{K}]\\
&  \geq \mathbb{\hat{E}}[(\mathbb{\hat{E}}_{t}[X_{n}]+\delta-\mathbb{\hat{E}%
}_{t}[X_{n}])^{+}I_{K}]\\
&  =\mathbb{\hat{E}}[\delta I_{K}]=\delta c(K)>0,
\end{align*}
which contradicts to $X\geq X_{n}$ q.s.. Thus $\mathbb{\hat{E}}_{t}[X_{n}]\leq
X$ q.s. for each $n\geq1$.

Let $\eta=\lim_{n\rightarrow \infty}\mathbb{\hat{E}}_{t}[X_{n}]\leq X$ q.s..
Now we show that $\eta=X$ q.s.. Otherwise $c(\{ \eta<X\})>0$. Obviously
$c(\{X=-\infty \})=0$. It is easy to check that
\[
\{ \eta+\frac{1}{k}<X\} \cap \{|\eta|<k\} \cap \{X_{1}>-k\} \uparrow \{ \eta<X\}
\text{ as }k\rightarrow \infty \text{.}%
\]
By (3) and (5) in Proposition \ref{pro2.8}, we can choose a constant $k>0$ and
a compact set $K\in \mathcal{F}_{t}$ such that $K\subset \{ \eta+\frac{1}{k}<X\}
\cap \{|\eta|<k\} \cap \{X_{1}>-k\}$ and $c(K)>0$. Noting that $(X_{n}%
+k)I_{K}\uparrow(X+k)I_{K}$ q.s., then we have
\[
\mathbb{\hat{E}}[(X_{n}+k)I_{K}]\uparrow \mathbb{\hat{E}}[(X+k)I_{K}%
]\geq \mathbb{\hat{E}}[(\eta+\frac{1}{k}+k)I_{K}].
\]

On the other hand, by Lemma \ref{le3.1},
\[
\mathbb{\hat{E}}[(X_{n}+k)I_{K}]=\mathbb{\hat{E}}[\mathbb{\hat{E}}_{t}%
[(X_{n}+k)^{+}]I_{K}]=\mathbb{\hat{E}}[(\mathbb{\hat{E}}_{t}[X_{n}%
]+k)I_{K}]\uparrow \mathbb{\hat{E}}[(\eta+k)I_{K}].
\]
which implies a contradiction by Proposition \ref{pro4}.
\end{proof}

\subsection{Extension under the norm $\mathbb{L}^{1}$}

We set%
\begin{align*}
\bar{L}_{G}^{1_{\ast}^{\ast}}(\Omega)  &  =\{X\in \mathbb{L}^{1}(\Omega
):\exists X_{n}\in L_{G}^{1_{\ast}^{\ast}}(\Omega)\text{ such that
}\mathbb{\hat{E}}[|X_{n}-X|]\rightarrow0\},\\
\bar{L}_{G}^{1_{\ast}^{\ast}}(\Omega_{t})  &  =\{X\in \mathbb{L}^{1}(\Omega
_{t}):\exists X_{n}\in L_{G}^{1_{\ast}^{\ast}}(\Omega_{t})\text{ such that
}\mathbb{\hat{E}}[|X_{n}-X|]\rightarrow0\}.
\end{align*}
In this subsection, we extend the conditional $G$-expectation to $\bar{L}%
_{G}^{1_{\ast}^{\ast}}(\Omega)$. For this we need the following representation
theorem. In the following, we suppose that $\mathcal{P}=\mathcal{P}_{\max}$.

\begin{proposition}
\label{pro3.2} For each $X\in L_{G}^{1_{\ast}^{\ast}}(\Omega)$, we have, for
each $P\in \mathcal{P}$,%
\[
\mathbb{\hat{E}}_{t}[X]=\underset{Q\in \mathcal{P}(t,P)}{ess\sup}^{P}%
E_{Q}[X|\mathcal{F}_{t}]\text{ \  \ }P\text{-a.s.}.
\]

\end{proposition}

\begin{proof}
We divide the proof into two part.

Step 1. For each fixed $X\in L_{G}^{1^{\ast}}(\Omega)$, there exists a
sequence $\{X_{n}\}_{n=1}^{\infty}\subset L_{G}^{1}(\Omega)$ such that
$X_{n}\downarrow X$ q.s.. Noting that $\mathbb{\hat{E}}_{t}[X_{n}%
]\downarrow \mathbb{\hat{E}}_{t}[X]$ q.s. and $E_{Q}[X_{n}|\mathcal{F}%
_{t}]\downarrow E_{Q}[X|\mathcal{F}_{t}]$ $P$-a.s., then, by Theorem
\ref{the2.15}, we obtain%
\[
\mathbb{\hat{E}}_{t}[X]\geq \underset{Q\in \mathcal{P}(t,P)}{ess\sup}^{P}%
E_{Q}[X|\mathcal{F}_{t}]\text{ \  \ }P\text{-a.s.}.
\]
Now we prove that $\mathbb{\hat{E}}_{t}[X]=\underset{Q\in \mathcal{P}%
(t,P)}{ess\sup}^{P}E_{Q}[X|\mathcal{F}_{t}]$ \  \ $P$-a.s.. Otherwise,
$E_{P}[\mathbb{\hat{E}}_{t}[X]]>E_{P}[\underset{Q\in \mathcal{P}(t,P)}{ess\sup
}^{P}E_{Q}[X|\mathcal{F}_{t}]]$. By Theorem \ref{the2.15} and Lemma
\ref{le2.24}, we get%
\begin{align*}
E_{P}[\mathbb{\hat{E}}_{t}[X]]  &  =\lim_{n\rightarrow \infty}E_{P}%
[\mathbb{\hat{E}}_{t}[X_{n}]]=\lim_{n\rightarrow \infty}E_{P}[\underset
{Q\in \mathcal{P}(t,P)}{ess\sup}^{P}E_{Q}[X_{n}|\mathcal{F}_{t}]]\\
&  =\lim_{n\rightarrow \infty}\underset{Q\in \mathcal{P}(t,P)}{\sup}E_{Q}[X_{n}]
\end{align*}
and%
\[
E_{P}[\underset{Q\in \mathcal{P}(t,P)}{ess\sup}^{P}E_{Q}[X|\mathcal{F}%
_{t}]]=\underset{Q\in \mathcal{P}(t,P)}{\sup}E_{Q}[X].
\]
Noting that $\mathcal{P}(t,P)$ is weakly compact, $\{X_{n}\}_{n=1}^{\infty
}\subset L_{G}^{1}(\Omega)$ and $X_{n}\downarrow X$ q.s., then, by Theorem
\ref{the2.10}, we can get $\underset{Q\in \mathcal{P}(t,P)}{\sup}E_{Q}%
[X_{n}]\downarrow \underset{Q\in \mathcal{P}(t,P)}{\sup}E_{Q}[X]$, which induces
a contradiction.

Step 2. For each fixed $X\in L_{G}^{1_{\ast}^{\ast}}(\Omega)$, there exists a
sequence $\{X_{n}\}_{n=1}^{\infty}\subset L_{G}^{1^{\ast}}(\Omega)$ such that
$X_{n}\uparrow X$ q.s.. Noting that $\mathbb{\hat{E}}_{t}[X_{n}]\uparrow
\mathbb{\hat{E}}_{t}[X]$ q.s. and $E_{Q}[X_{n}|\mathcal{F}_{t}]\uparrow
E_{Q}[X|\mathcal{F}_{t}]$ $P$-a.s., then, by Step 1, we get $\mathbb{\hat{E}%
}_{t}[X]\geq \underset{Q\in \mathcal{P}(t,P)}{ess\sup}^{P}E_{Q}[X|\mathcal{F}%
_{t}]$ \  \ $P$-a.s.. Now we assert $\mathbb{\hat{E}}_{t}[X]=\underset
{Q\in \mathcal{P}(t,P)}{ess\sup}^{P}E_{Q}[X|\mathcal{F}_{t}]$ \  \ $P$-a.s..
Otherwise, $E_{P}[\mathbb{\hat{E}}_{t}[X]]>E_{P}[\underset{Q\in \mathcal{P}%
(t,P)}{ess\sup}^{P}E_{Q}[X|\mathcal{F}_{t}]]$. Similar to the analysis of Step
1, we can obtain%
\[
\lim_{n\rightarrow \infty}\underset{Q\in \mathcal{P}(t,P)}{\sup}E_{Q}%
[X_{n}]>\underset{Q\in \mathcal{P}(t,P)}{\sup}E_{Q}[X]\text{,}%
\]
which contradicts to $\underset{Q\in \mathcal{P}(t,P)}{\sup}E_{Q}[X_{n}%
]\leq \underset{Q\in \mathcal{P}(t,P)}{\sup}E_{Q}[X]$ for $n\geq1$. The proof is complete.
\end{proof}

\begin{proposition}
\label{pro3.3} For each $X$, $Y\in L_{G}^{1_{\ast}^{\ast}}(\Omega)$, we have%
\[
\mathbb{\hat{E}}[|\mathbb{\hat{E}}_{t}[X]-\mathbb{\hat{E}}_{t}[Y]|]\leq
\mathbb{\hat{E}}[|X-Y|].
\]

\end{proposition}

\begin{proof}
For each $P\in \mathcal{P}$, by Proposition \ref{pro3.2} and Lemma
\ref{le2.24}, we get%
\begin{align*}
E_{P}[|\mathbb{\hat{E}}_{t}[X]-\mathbb{\hat{E}}_{t}[Y]|]  &  =E_{P}%
[|\underset{Q\in \mathcal{P}(t,P)}{ess\sup}^{P}E_{Q}[X|\mathcal{F}%
_{t}]-\underset{Q\in \mathcal{P}(t,P)}{ess\sup}^{P}E_{Q}[Y|\mathcal{F}_{t}]|]\\
&  \leq E_{P}[\underset{Q\in \mathcal{P}(t,P)}{ess\sup}^{P}|E_{Q}%
[X|\mathcal{F}_{t}]-E_{Q}[Y|\mathcal{F}_{t}]|]\\
&  \leq E_{P}[\underset{Q\in \mathcal{P}(t,P)}{ess\sup}^{P}E_{Q}%
[|X-Y||\mathcal{F}_{t}]\\
&  =\underset{Q\in \mathcal{P}(t,P)}{\sup}E_{Q}[|X-Y|]\leq \mathbb{\hat{E}%
}[|X-Y|].
\end{align*}

\end{proof}

Now we give the extension of the conditional $G$-expectation to $\bar{L}%
_{G}^{1_{\ast}^{\ast}}(\Omega)$.

\begin{definition}
\label{de3.4} For each $X\in \bar{L}_{G}^{1_{\ast}^{\ast}}(\Omega)$, there
exists a sequence $\{X_{n}\}_{n=1}^{\infty}\subset L_{G}^{1_{\ast}^{\ast}%
}(\Omega)$ such that $\mathbb{\hat{E}}[|X_{n}-X|]\rightarrow0$. We define%
\[
\mathbb{\hat{E}}_{t}[X]=\mathbb{L}^{1}-\lim_{n\rightarrow \infty}%
\mathbb{\hat{E}}_{t}[X_{n}].
\]

\end{definition}

By Proposition \ref{pro3.3}, it is easy to show that the above definition is
meaningful, $\mathbb{\hat{E}}_{t}[\cdot]:\bar{L}_{G}^{1_{\ast}^{\ast}}%
(\Omega)\rightarrow \bar{L}_{G}^{1_{\ast}^{\ast}}(\Omega_{t})$ and the
following properties.

\begin{proposition}
\label{pro3.5} We have

\begin{description}
\item[(1)] $X$, $Y\in \bar{L}_{G}^{1_{\ast}^{\ast}}(\Omega)$, $X\geq
Y\Longrightarrow \mathbb{\hat{E}}_{t}[X]\geq \mathbb{\hat{E}}_{t}[Y]$;

\item[(2)] $X\in \bar{L}_{G}^{1_{\ast}^{\ast}}(\Omega_{t})$, $Y\in \bar{L}%
_{G}^{1_{\ast}^{\ast}}(\Omega)\Longrightarrow \mathbb{\hat{E}}_{t}%
[X+Y]=X+\mathbb{\hat{E}}_{t}[Y]$;

\item[(3)] $X$, $Y\in \bar{L}_{G}^{1_{\ast}^{\ast}}(\Omega)\Longrightarrow
\mathbb{\hat{E}}_{t}[X+Y]\leq \mathbb{\hat{E}}_{t}[X]+\mathbb{\hat{E}}_{t}[Y]$;

\item[(4)] $X\in \bar{L}_{G}^{1_{\ast}^{\ast}}(\Omega_{t})$ is bounded,
$X\geq0$, $Y\in \bar{L}_{G}^{1_{\ast}^{\ast}}(\Omega)$, $Y\geq0$,
$\lim_{n\rightarrow \infty}\mathbb{\hat{E}}[YI_{\{Y\geq n\}}]=0\Longrightarrow
\mathbb{\hat{E}}_{t}[XY]=X\mathbb{\hat{E}}_{t}[Y]$;

\item[(5)] $X\in \bar{L}_{G}^{1_{\ast}^{\ast}}(\Omega)\Longrightarrow
\mathbb{\hat{E}}_{s}[\mathbb{\hat{E}}_{t}[X]]=\mathbb{\hat{E}}_{s\wedge t}[X]$.
\end{description}
\end{proposition}

\begin{proposition}
\label{ne-pro3.5} For each $X\in \bar{L}_{G}^{1_{\ast}^{\ast}}(\Omega)$, we
have, for each $P\in \mathcal{P}$,%
\[
\mathbb{\hat{E}}_{t}[X]=\underset{Q\in \mathcal{P}(t,P)}{ess\sup}^{P}%
E_{Q}[X|\mathcal{F}_{t}]\text{ \  \ }P\text{-a.s.}.
\]

\end{proposition}

\begin{proof}
We omit the proof which is similar to Theorem \ref{the2.15}.
\end{proof}

It is important to note that $L_{G}^{1^{\ast}}(\Omega)$, $L_{G}^{1_{\ast
}^{\ast}}(\Omega)$ and $\bar{L}_{G}^{1_{\ast}^{\ast}}(\Omega)$ are not linear
spaces. In the following, we consider a linear space in $L_{G}^{1_{\ast}%
^{\ast}}(\Omega)$. We set%
\[
L_{G}^{{\ast}1}(\Omega)=\{X-Y:X,Y\in L_{G}^{1^{\ast}}(\Omega)\}.
\]
It is easy to check that $L_{G}^{1^{\ast}}(\Omega)\subset L_{G}^{{\ast}1}%
(\Omega)\subset L_{G}^{1_{\ast}^{\ast}}(\Omega)$ and $L_{G}^{{\ast}1}(\Omega)$ is a linear space.

\begin{proposition}
If $\xi \in L_{G}^{{\ast}1}(\Omega)$, then $|\xi|\in L_{G}^{{\ast}1}(\Omega)$.
\end{proposition}

\begin{proof}
For $\xi=X-Y$, where $X$, $Y\in L_{G}^{1^{\ast}}(\Omega)$, it is easy to check
that $|\xi|=|X-Y|=X\vee Y-X\wedge Y$, thus $|\xi|\in L_{G}^{{\ast}1}(\Omega)$.
\end{proof}

We also set%
\[
\bar{L}_{G}^{{\ast}1}(\Omega)=\{X\in \mathbb{L}^{1}(\Omega):\exists
X_{n}\in L_{G}^{{\ast}1}(\Omega)\text{ such that }\mathbb{\hat{E}}%
[|X_{n}-X|]\rightarrow0\}.
\]
It is easy to verify that $\bar{L}_{G}^{{\ast}1}(\Omega)\subset \bar{L}%
_{G}^{1_{\ast}^{\ast}}(\Omega)$ and $\bar{L}_{G}^{{\ast}1}(\Omega)$ is a
vector lattice.

\subsection{Application to optional stopping theorem}

In this subsection, we first present some basic properties  in the extended spaces
$L_{G}^{1^{\ast}}(\Omega)$, $L_{G}^{{\ast}1}(\Omega)$ and $L_{G}^{1_{\ast
}^{\ast}}(\Omega)$.

\begin{proposition}
\label{pro3.6} We have

\begin{description}
\item[(1)] Let $X$ be a bounded upper (resp. lower) semicontinuous function on
$\Omega$. Then $X\in L_{G}^{1^{\ast}}(\Omega)$ (resp. $X\in L_{G}^{{\ast}1}(\Omega)$);

\item[(2)] Let $X\in L_{G}^{1}(\Omega;\mathbb{R}^{n})$ and let $f$ be a
bounded upper (resp. lower) semicontinuous function on $\mathbb{R}^{n}$. Then
$f(X)\in L_{G}^{1^{\ast}}(\Omega)$ (resp. $f(X)\in L_{G}^{{\ast}1}%
(\Omega)$);

\item[(3)] Let $X\in L_{G}^{1}(\Omega;\mathbb{R}^{n})$ and let $f_{i}$,
$i\geq1$, be a sequence of uniformly bounded upper or lower semicontinuous
functions on $\mathbb{R}^{n}$ such that $f_{i}\uparrow f$. Then $f(X)\in
L_{G}^{1_{\ast}^{\ast}}(\Omega)$.
\end{description}
\end{proposition}

\begin{proof}
We only prove (2). (1) and (3) can be proved similarly. For bounded upper
semicontinuous function $f$, there exists a sequence bounded continuous
functions $\varphi_{i}$ such that $\varphi_{i}\downarrow f$. It is easy to
check that $\varphi_{i}(X)\in L_{G}^{1}(\Omega)$ and $\varphi_{i}(X)\downarrow
f(X)$. Thus $f(X)\in L_{G}^{1^{\ast}}(\Omega)$. Similar analysis for lower
semicontinuous function.
\end{proof}

\begin{remark}
Let $X\in L_{G}^{1}(\Omega)$, $-\infty<a<b<\infty$. Then by the above
proposition, $I_{\{X\leq a\}}$, $I_{\{X\geq a\}}$, $I_{\{a\leq X\leq b\}}\in
L_{G}^{1^{\ast}}(\Omega)$, $I_{\{X<a\}}$, $I_{\{a<X<b\}}$, $I_{\{a\leq
X<b\}}\in L_{G}^{{\ast}1}(\Omega)$.
\end{remark}

\begin{proposition}
\label{pro8}Let $X_{n}\in L_{G}^{1^{\ast}}(\Omega)$ and $X_{n}\geq Y$ for
$n\geq1$, where $Y\in \mathbb{L}^{1}(\Omega)$. Then $\lim \inf_{n\rightarrow
\infty}X_{n}\in \mathcal{L}_{G}^{1_{\ast}^{\ast}}(\Omega)$.
\end{proposition}

\begin{proof}
It is simply due to
\[
\underset{n\rightarrow \infty}{\liminf}X_{n}=\sup_{n}(\inf_{k\geq n}X_{k}),
\]
and $\wedge_{k=n}^{n+m}X_{k}\downarrow \inf_{k\geq n}X_{k}\in L_{G}^{1^{\ast}%
}(\Omega)$ as $m\rightarrow \infty$.
\end{proof}

We now give the definition for a class of stopping times.

\begin{definition}
\label{def3.7} A random time $\tau:\Omega \rightarrow \lbrack0,\infty)$ is
called a $*$-stopping time if $I_{\{ \tau \geq t\}}\in L_{G}^{1^{\ast}}(\Omega_{t})$
for each $t\geq0$.
\end{definition}

\begin{definition}
\label{def3.8} For a given $*$-stopping time $\tau$ and $\xi \in \bar{L}_{G}^{1_{\ast}%
^{\ast}}(\Omega)$, we define $\mathbb{\hat{E}}_{\tau}[\xi]=M_{\tau}$, where
$M_{t}=\mathbb{\hat{E}}_{t}[\xi]$ for $t\geq0$.
\end{definition}

\begin{example}
\label{ex3.9} Let $(X_{t})_{t\leq T}$ be an $n$-dimensional right continuous
process such that $X_{t}\in L_{G}^{1}(\Omega_{t};\mathbb{R}^{n})$ for $t\leq
T$. We define a random time for each fixed closed set $F\subset \mathbb{R}^{n}$
as follows:%
\[
\tau=\inf \{t\geq0:X_{t}\not \in F\} \wedge T.
\]
It is easy to check that
\[
\{ \tau \geq t\}=%
{\displaystyle \bigcap \limits_{s\in \lbrack0,t)\cap \mathbb{Q}}}
\{X_{s}\in F\}=%
{\displaystyle \bigcap \limits_{s\in \lbrack0,t)\cap \mathbb{Q}}}
\{d(X_{s},F)=0\},
\]
and%
\[
I_{\{ \tau \geq t\}}=\inf \{I_{\{d(X_{s},F)=0\}}:s\in \lbrack0,t)\cap \mathbb{Q}\},
\]
where $d(x,F):=\inf_{\{y\in F\}}|x-y|$. By Propositions \ref{pro3} and \ref{pro3.6}, we know
that $I_{\{ \tau \geq t\}}\in L_{G}^{1^{\ast}}(\Omega_{t})$. Thus $\tau$ is a
$*$-stopping time.
\end{example}

\begin{proposition}
\label{pro3.10} Let $\tau$ be a $*$-stopping time. Then $\tau \wedge T\in
L_{G}^{1^{\ast}}(\Omega)$ for each given $T>0$ and $\tau \in \mathcal{L}%
_{G}^{1_{\ast}^{\ast}}(\Omega)$.
\end{proposition}

\begin{proof}
For each $n\geq1$, we define%
\begin{align*}
\tau^{n} &  :=\sum_{i=1}^{2^{n}}iT2^{-n}I_{\{(i-1)T2^{-n}\leq \tau<iT2^{-n}%
\}}+TI_{\{ \tau \geq T\}}\\
&  =\sum_{i=1}^{2^{n}}T2^{-n}I_{\{(i-1)T2^{-n}\leq \tau \}}\in L_{G}^{1^{\ast}%
}(\Omega).
\end{align*}
Since $\tau^{n}\downarrow$ $\tau \wedge T$ and $\tau \wedge m\uparrow \tau$ as
$m\rightarrow \infty$, we have  $\tau \wedge T\in L_{G}^{1^{\ast}}(\Omega)$ and
$\tau \in \mathcal{L}_{G}^{1_{\ast}^{\ast}}(\Omega)$.
\end{proof}

We now prove the following optional stopping theorem.

\begin{theorem}
\label{the3.11} Suppose that there exists some $\underline{\sigma}^{2}>0$ such
that $G(A)-G(B)\geq \underline{\sigma}^{2}\mathrm{tr}[A-B]$ for any $A\geq B$.
Let $M_{t}=\mathbb{\hat{E}}_{t}[\xi]$ for $t\leq T$, $\xi \in L_{G}^{p}%
(\Omega_{T})$ with $p>1$, and let $\sigma$, $\tau$ be two $*$-stopping times with
$\sigma \leq \tau \leq T$. Then $M_{\tau}$, $M_{\sigma}\in \bar{L}_{G}^{{\ast}1}(\Omega_{T})$ and
\[
M_{\sigma}=\mathbb{\hat{E}}_{\sigma}[M_{\tau}]\text{ \ q.s.}.
\]

\end{theorem}

\begin{proof}
By Definition \ref{def3.8}, we only need to prove that $M_{\tau}\in \bar{L}%
_{G}^{{\ast}1}(\Omega_{T})$ and $\mathbb{\hat{E}}_{t}[M_{\tau}%
]=M_{t\wedge \tau}$ \ q.s..

Step 1. Suppose that there exists a $L>0$ such that $|\xi|\leq L$. For each
given $n\geq1$, we set $N=2^{n}$, $t_{i}=iT2^{-n}$, $i\leq N$, and define%
\[
\tau^{n}:=\sum_{i=1}^{N}t_{i}I_{\{t_{i-1}\leq \tau<t_{i}\}}+TI_{\{ \tau=T\}}.
\]
It is easy to check that
\begin{align*}
M_{\tau^{n}}  & =\sum_{i=1}^{N}M_{t_{i}}I_{\{t_{i-1}\leq \tau<t_{i}\}}%
+M_{T}I_{\{ \tau=T\}}=M_{t_{1}}+\sum_{i=2}^{N}(M_{t_{i}}-M_{t_{i-1}}%
)I_{\{ \tau \geq t_{i-1}\}}\\
& =M_{t_{1}}+\sum_{i=2}^{N}(M_{t_{i}}-M_{t_{i-1}})^{+}I_{\{ \tau \geq t_{i-1}%
\}}-\sum_{i=2}^{N}(M_{t_{i}}-M_{t_{i-1}})^{-}I_{\{ \tau \geq t_{i-1}\}}.
\end{align*}
Thus $M_{\tau^{n}}\in L_{G}^{{\ast}1}(\Omega_{T})$ and
\begin{align*}
\mathbb{\hat{E}}_{t_{N-1}}[M_{\tau^{n}}]  & =\mathbb{\hat{E}}_{t_{N-1}}%
[\eta+(M_{t_{N}}-M_{t_{N-1}}+2L)I_{\{ \tau \geq t_{N-1}\}}-2LI_{\{ \tau \geq
t_{N-1}\}}]\\
& =\eta-2LI_{\{ \tau \geq t_{N-1}\}}+\mathbb{\hat{E}}_{t_{N-1}}[M_{t_{N}%
}-M_{t_{N-1}}+2L]I_{\{ \tau \geq t_{N-1}\}}\\
& =\eta,
\end{align*}
where $\eta=M_{t_{1}}+\sum_{i=2}^{N-1}(M_{t_{i}}-M_{t_{i-1}})I_{\{ \tau \geq
t_{i-1}\}}$. Repeat this process, we can get%
\[
\mathbb{\hat{E}}_{t}[M_{\tau^{n}}]=M_{t_{1}\wedge t}+\sum_{i=2}^{N}%
(M_{t_{i}\wedge t}-M_{t_{i-1}\wedge t})I_{\{ \tau \geq t_{i-1}\}}=M_{t\wedge
\tau^{n}}.
\]
By Theorem 4.5 in Song \cite{Song11} (or Theorem 5.1 in \cite{STZ}), there
exist a $Z\in H_{G}^{1}(0,T;\mathbb{R}^{d})$ and a decreasing $G$-martingale $K$ with
$K_{0}=0$, $K_{T}\in L_{G}^{1}(\Omega_{T})$ such that%
\[
M_{t}=\mathbb{\hat{E}}[\xi]+\int_{0}^{t}Z_{s}dB_{s}+K_{t}.
\]
Similar to $M_{\tau^{n}}$, we can get $\int_{0}^{\tau^{n}}Z_{s}dB_{s}\in
L_{G}^{{\ast}1}(\Omega_{T})$ and $-K_{\tau^{n}}\in L_{G}^{1^{\ast}}%
(\Omega_{T})$. By the B-D-G inequality, we can get%
\begin{align*}
\mathbb{\hat{E}}[|\int_{\tau}^{\tau^{n}}Z_{s}dB_{s}|]  & \leq C\mathbb{\hat
{E}}[(\int_{\tau}^{\tau^{n}}|Z_{s}|^{2}ds)^{1/2}]\\
& \leq C\mathbb{\hat{E}}[(\int_{\tau}^{\tau^{n}}(|Z_{s}|\wedge N_{1}%
)^{2}ds)^{1/2}+(\int_{\tau}^{\tau^{n}}|Z_{s}|^{2}I_{\{|Z_{s}|\geq N_{1}%
\}}ds)^{1/2}]\\
& \leq C(N_{1}\sqrt{T2^{-n}}+\mathbb{\hat{E}}[(\int_{0}^{T}|Z_{s}%
|^{2}I_{\{|Z_{s}|\geq N_{1}\}}ds)^{1/2}]).
\end{align*}
Thus we can get $\underset{n\rightarrow \infty}{\lim \sup}\mathbb{\hat{E}}%
[|\int_{\tau}^{\tau^{n}}Z_{s}dB_{s}|]\leq C\mathbb{\hat{E}}[(\int_{0}%
^{T}|Z_{s}|^{2}I_{\{|Z_{s}|\geq N_{1}\}}ds)^{1/2}]$ for each $N_{1}>0$. Since
$Z\in H_{G}^{1}(0,T;\mathbb{R}^{d})$, using the same analysis as in
Proposition 18 in \cite{DHP11}, we can obtain $\mathbb{\hat{E}}[(\int_{0}%
^{T}|Z_{s}|^{2}I_{\{|Z_{s}|\geq N_{1}\}}ds)^{1/2}]\rightarrow0$ as
$N_{1}\rightarrow \infty$, which implies $\mathbb{\hat{E}}[|\int_{\tau}%
^{\tau^{n}}Z_{s}dB_{s}|]\rightarrow0$ as $n\rightarrow \infty$. Thus $\int
_{0}^{\tau}Z_{s}dB_{s}\in \bar{L}_{G}^{{\ast}1}(\Omega_{T})$. Noting that
$K_{\tau^{n}}\uparrow K_{\tau}$, then we get $-K_{\tau}\in L_{G}^{1^{\ast}%
}(\Omega_{T})$. Thus $M_{\tau}\in \bar{L}_{G}^{{\ast}1}(\Omega_{T})$ and
\[
M_{t\wedge \tau^{n}}=\mathbb{\hat{E}}_{t}[M_{\tau^{n}}]\leq \mathbb{\hat{E}}%
_{t}[\mathbb{\hat{E}}[\xi]+\int_{0}^{\tau^{n}}Z_{s}dB_{s}+K_{\tau}].
\]
Noting that $M_{t\wedge \tau^{n}}\rightarrow M_{t\wedge \tau}$ q.s. and
$\mathbb{\hat{E}}_{t}[M_{\tau}]=\mathbb{L}^{1}-\lim_{n\rightarrow \infty
}\mathbb{\hat{E}}_{t}[\mathbb{\hat{E}}[\xi]+\int_{0}^{\tau^{n}}Z_{s}%
dB_{s}+K_{\tau}]$, then we obtain $M_{t\wedge \tau}\leq \mathbb{\hat{E}}%
_{t}[M_{\tau}]$. On the other hand, by $K_{\tau}\leq K_{t\wedge \tau}$ and
$-K_{t\wedge \tau}\in L_{G}^{1^{\ast}}(\Omega_{t})$, we obtain
\[
\mathbb{\hat{E}}_{t}[M_{\tau}]\leq \mathbb{\hat{E}}_{t}[\mathbb{\hat{E}}%
[\xi]+\int_{0}^{\tau}Z_{s}dB_{s}+K_{t\wedge \tau}]=\mathbb{\hat{E}}%
[\xi]+K_{t\wedge \tau}+\mathbb{\hat{E}}_{t}[\int_{0}^{\tau}Z_{s}dB_{s}].
\]
Using the above method, we can get $\mathbb{\hat{E}}_{t}[\int_{0}^{\tau^{n}%
}Z_{s}dB_{s}]=\int_{0}^{t\wedge \tau^{n}}Z_{s}dB_{s}$, and then $\mathbb{\hat
{E}}_{t}[\int_{0}^{\tau}Z_{s}dB_{s}]=\int_{0}^{t\wedge \tau}Z_{s}dB_{s}$, which
implies $\mathbb{\hat{E}}_{t}[M_{\tau}]\leq M_{t\wedge \tau}$. Thus
$\mathbb{\hat{E}}_{t}[M_{\tau}]=M_{t\wedge \tau}$.

Step 2. For $\xi \in L_{G}^{p}(\Omega_{T})$ with $p>1$, we set $\xi^{n}%
=(\xi \wedge n)\vee(-n)$ and $M_{t}^{n}=\mathbb{\hat{E}}_{t}[\xi^{n}]$. By Step
1, we get $\mathbb{\hat{E}}_{t}[M_{\tau}^{n}]=M_{t\wedge \tau}^{n}$. By Theorem
3.4 in \cite{Song11} and Proposition 3.9 in \cite{HJPS}, we can obtain
$\mathbb{\hat{E}}[\sup_{t\leq T}|M_{t}^{n}-M_{t}|]\rightarrow0$ as
$n\rightarrow \infty$. Thus $\mathbb{\hat{E}}[|M_{\tau}^{n}-M_{\tau
}|]\rightarrow0$, which implies $M_{\tau}\in \bar{L}_{G}^{{\ast}1}%
(\Omega_{T})$ and $\mathbb{\hat{E}}_{t}[M_{\tau}]=\mathbb{L}^{1}%
-\lim_{n\rightarrow \infty}\mathbb{\hat{E}}_{t}[M_{\tau}^{n}]=\mathbb{L}%
^{1}-\lim_{n\rightarrow \infty}M_{t\wedge \tau}^{n}=M_{t\wedge \tau}$. The proof
is complete.
\end{proof}

\section{Extension of nonlinear expectations}

Let $(\Omega,L_{ip}(\Omega),\mathbb{\hat{E}}[\cdot])$ be the $G$-expectation
space and let $(\Omega,L_{ip}(\Omega),(\mathbb{\tilde{E}}_{t}[\cdot])_{t\geq
0})$ be a consistent nonlinear expectation satisfying the following properties:

\begin{description}
\item[(1)] $\mathbb{\tilde{E}}_{t}[\cdot]:L_{ip}(\Omega)\rightarrow
L_{ip}(\Omega_{t})$;

\item[(2)] $X$, $Y\in L_{ip}(\Omega)$, $X\leq Y\Longrightarrow \mathbb{\tilde
{E}}_{t}[X]\leq \mathbb{\tilde{E}}_{t}[Y]$;

\item[(3)] $X\in L_{ip}(\Omega_{t})$, $Y\in L_{ip}(\Omega)\Longrightarrow
\mathbb{\tilde{E}}_{t}[X+Y]=X+\mathbb{\tilde{E}}_{t}[Y]$;

\item[(4)] $X\in L_{ip}(\Omega)\Longrightarrow \mathbb{\tilde{E}}%
_{s}[\mathbb{\tilde{E}}_{t}[X]]=\mathbb{\tilde{E}}_{s\wedge t}[X]$;

\item[(5)] $X$, $Y\in L_{ip}(\Omega)\Longrightarrow \mathbb{\tilde{E}}%
_{t}[X]-\mathbb{\tilde{E}}_{t}[Y]\leq \mathbb{\hat{E}}_{t}[X-Y]$.
\end{description}

We have the following proposition.

\begin{proposition}
\label{pro4.1}(\cite{P10}) For each $X$, $Y\in L_{ip}(\Omega)$, we have
$\mathbb{\tilde{E}}_{t}[X]\leq \mathbb{\hat{E}}_{t}[X]$ and $|\mathbb{\tilde
{E}}_{t}[X]-\mathbb{\tilde{E}}_{t}[Y]|\leq \mathbb{\hat{E}}_{t}[|X-Y|]$.
\end{proposition}

By this proposition, $\mathbb{\tilde{E}}_{t}[\cdot]$ can be extended to
$L_{G}^{1}(\Omega)$. In the following, we consider the extension of
$\mathbb{\tilde{E}}_{t}[\cdot]$.

\begin{definition}
\label{de4.2} For each $X\in L_{G}^{1^{\ast}}(\Omega)$, there exists a
sequence $\{X_{n}\}_{n=1}^{\infty}\subset L_{G}^{1}(\Omega)$ such that
$X_{n}\downarrow X$ q.s., we define%
\[
\mathbb{\tilde{E}}_{t}[X]=\lim_{n\rightarrow \infty}\mathbb{\tilde{E}}%
_{t}[X_{n}]\text{ q.s..}%
\]

\end{definition}

The following proposition show that the above definition is meaningful.

\begin{proposition}
\label{pro4.3}Let $X\in L_{G}^{1^{\ast}}(\Omega)$ and let $\{X_{n}%
\}_{n=1}^{\infty}$, $\{ \tilde{X}_{n}\}_{n=1}^{\infty}$ be two sequences in
$L_{G}^{1}(\Omega)$ such that $X_{n}\downarrow X$ and $\tilde{X}_{n}\downarrow
X$ q.s.. Then
\[
\lim_{n\rightarrow \infty}\mathbb{\tilde{E}}_{t}[X_{n}]=\lim_{n\rightarrow
\infty}\mathbb{\tilde{E}}_{t}[\tilde{X}_{n}],\  \  \text{q.s..}%
\]

\end{proposition}

\begin{proof}
For each fixed $m$,
\[
\mathbb{\tilde{E}}_{t}[X_{n}]-\mathbb{\tilde{E}}_{t}[\tilde{X}_{m}%
]\leq \mathbb{\hat{E}}_{t}[X_{n}-\tilde{X}_{m}]\leq \mathbb{\hat{E}}_{t}%
[(X_{n}-\tilde{X}_{m})^{+}].
\]
Since $X_{n}\downarrow X$ q.s., we get $(X_{n}-\tilde{X}_{m})^{+}\downarrow0$
q.s.. By Lemma \ref{le1}, we obtain $\lim_{n\rightarrow \infty}\mathbb{\tilde
{E}}_{t}[X_{n}]$ $\leq \mathbb{\tilde{E}}_{t}[\tilde{X}_{m}]$ q.s.. Thus
$\lim_{n\rightarrow \infty}\mathbb{\tilde{E}}_{t}[X_{n}]\leq \lim_{n\rightarrow
\infty}\mathbb{\tilde{E}}_{t}[\tilde{X}_{n}]$ q.s.. Similarly, we can prove
that $\lim_{n\rightarrow \infty}\mathbb{\tilde{E}}_{t}[X_{n}]\geq
\lim_{n\rightarrow \infty}\mathbb{\tilde{E}}_{t}[\tilde{X}_{n}]$ q.s.. The
proof is complete.
\end{proof}

\begin{proposition}
\label{pro4.4} We have

\begin{description}
\item[(1)] $X,Y\in L_{G}^{1^{\ast}}(\Omega)$, $X\leq Y\ $q.s.$\Longrightarrow
\mathbb{\tilde{E}}_{t}[X]\leq \mathbb{\tilde{E}}_{t}[Y]\  \ $q$_{.}$s$_{.}$;

\item[(2)] $\{X_{n}\}_{n=1}^{\infty}\subset L_{G}^{1^{\ast}}(\Omega)$,
$X_{n}\downarrow X\in L_{G}^{1^{\ast}}(\Omega)$ q.s.$\Longrightarrow
\mathbb{\tilde{E}}_{t}[X]=\lim_{n\rightarrow \infty}\mathbb{\tilde{E}}%
_{t}[X_{n}]$ q.s.;

\item[(3)] $X\in L_{G}^{1^{\ast}}(\Omega_{t})$, $Y\in L_{G}^{1^{\ast}}%
(\Omega)\Longrightarrow \mathbb{\tilde{E}}_{t}[X+Y]=X+\mathbb{\tilde{E}}%
_{t}[Y]$;

\item[(4)] $X\in L_{G}^{1^{\ast}}(\Omega)\Longrightarrow \mathbb{\tilde{E}}%
_{t}[X]\in L_{G}^{1^{\ast}}(\Omega_{t})$ and $\mathbb{\tilde{E}}%
_{s}[\mathbb{\tilde{E}}_{t}[X]]=\mathbb{\tilde{E}}_{s\wedge t}[X]$;

\item[(5)] $X\in L_{G}^{1^{\ast}}(\Omega)$, $Y\in L_{G}^{1^{\ast}}%
(\Omega)\Longrightarrow \mathbb{\tilde{E}}_{t}[X]-\mathbb{\tilde{E}}_{t}%
[Y]\leq \mathbb{\hat{E}}_{t}[X-Y]$.
\end{description}
\end{proposition}

\begin{proof}
The proof of (1)-(4) is similar to Proposition \ref{pro3}, we only prove (5).
Let $\{X_{n}\}_{n=1}^{\infty}$, $\{Y_{n}\}_{n=1}^{\infty}$ be two sequences in
$L_{G}^{1}(\Omega)$ such that $X_{n}\downarrow X$ and $Y_{n}\downarrow Y$
q.s.. For each $n$, $m\geq1$, $\mathbb{\tilde{E}}_{t}[X_{n}]-\mathbb{\tilde
{E}}_{t}[Y_{m}]\leq \mathbb{\hat{E}}_{t}[X_{n}-Y_{m}]$, then we get
$\mathbb{\tilde{E}}_{t}[X]-\mathbb{\tilde{E}}_{t}[Y_{m}]\leq \mathbb{\hat{E}%
}_{t}[X-Y_{m}]$ by taking $n\rightarrow \infty$. Noting that $X-Y_{m}\in
L_{G}^{1^{\ast}}(\Omega)$ and $X-Y_{m}\uparrow X-Y\in L_{G}^{1_{\ast}^{\ast}%
}(\Omega)$, then we obtain $\mathbb{\tilde{E}}_{t}[X]-\mathbb{\tilde{E}}%
_{t}[Y]\leq \mathbb{\hat{E}}_{t}[X-Y]$ by taking $m\rightarrow \infty$.
\end{proof}

\begin{definition}
\label{de4.5} For each $X\in-L_{G}^{1^{\ast}}(\Omega)$, there exists a
sequence $\{X_{n}\}_{n=1}^{\infty}\subset L_{G}^{1^{\ast}}(\Omega)$ such that
$X_{n}\uparrow X$ q.s., we define%
\[
\mathbb{\tilde{E}}_{t}[X]=\lim_{n\rightarrow \infty}\mathbb{\tilde{E}}%
_{t}[X_{n}]\text{ q.s..}%
\]

\end{definition}

The following proposition prove that Definition \ref{de4.5} is meaningful.

\begin{proposition}
\label{pro4.6} Let $\{X_{n}\}_{n=1}^{\infty}$, $\{ \tilde{X}_{n}%
\}_{n=1}^{\infty}$ be two sequences in $L_{G}^{1}(\Omega)$ such that
$X_{n}\uparrow X$ and $\tilde{X}_{n}\uparrow X$ q.s.. Then
\[
\lim_{n\rightarrow \infty}\mathbb{\tilde{E}}_{t}[X_{n}]=\lim_{n\rightarrow
\infty}\mathbb{\tilde{E}}_{t}[\tilde{X}_{n}],\  \  \text{q.s..}%
\]

\end{proposition}

\begin{proof}
For each fixed $m$,
\[
\mathbb{\tilde{E}}_{t}[\tilde{X}_{m}]-\mathbb{\tilde{E}}_{t}[X_{n}%
]\leq \mathbb{\hat{E}}_{t}[\tilde{X}_{m}-X_{n}]\leq \mathbb{\hat{E}}_{t}%
[(\tilde{X}_{m}-X_{n})^{+}].
\]
Since $X_{n}\uparrow X$ q.s., we get $(\tilde{X}_{m}-X_{n})^{+}\downarrow0$
q.s.. By Lemma \ref{le1}, we obtain $\lim_{n\rightarrow \infty}\mathbb{\tilde
{E}}_{t}[X_{n}]$ $\geq \mathbb{\tilde{E}}_{t}[\tilde{X}_{m}]$ q.s.. Thus
$\lim_{n\rightarrow \infty}\mathbb{\tilde{E}}_{t}[X_{n}]\geq \lim_{n\rightarrow
\infty}\mathbb{\tilde{E}}_{t}[\tilde{X}_{n}]$ q.s.. Similarly, we can prove
that $\lim_{n\rightarrow \infty}\mathbb{\tilde{E}}_{t}[X_{n}]\leq
\lim_{n\rightarrow \infty}\mathbb{\tilde{E}}_{t}[\tilde{X}_{n}]$ q.s.. The
proof is complete.
\end{proof}

\begin{proposition}
\label{pro4.7}We have

\begin{description}
\item[(1)] $X,Y\in-L_{G}^{1^{\ast}}(\Omega)$, $X\leq Y\ $q.s.$\Longrightarrow
\mathbb{\tilde{E}}_{t}[X]\leq \mathbb{\tilde{E}}_{t}[Y]\  \ $q$_{.}$s$_{.}$;

\item[(2)] $\{X_{n}\}_{n=1}^{\infty}\subset-L_{G}^{1^{\ast}}(\Omega)$,
$X_{n}\uparrow X\in-L_{G}^{1^{\ast}}(\Omega)$ q.s.$\Longrightarrow
\mathbb{\tilde{E}}_{t}[X]=\lim_{n\rightarrow \infty}\mathbb{\tilde{E}}%
_{t}[X_{n}]$ q.s.;

\item[(3)] $X\in-L_{G}^{1^{\ast}}(\Omega_{t})$, $Y\in-L_{G}^{1^{\ast}}%
(\Omega)\Longrightarrow \mathbb{\tilde{E}}_{t}[X+Y]=X+\mathbb{\tilde{E}}%
_{t}[Y]$;

\item[(4)] $X\in-L_{G}^{1^{\ast}}(\Omega)\Longrightarrow \mathbb{\tilde{E}}%
_{t}[X]\in-L_{G}^{1^{\ast}}(\Omega_{t})$ and $\mathbb{\tilde{E}}%
_{s}[\mathbb{\tilde{E}}_{t}[X]]=\mathbb{\tilde{E}}_{s\wedge t}[X]$;

\item[(5)] $X\in-L_{G}^{1^{\ast}}(\Omega)$, $Y\in-L_{G}^{1^{\ast}}%
(\Omega)\Longrightarrow \mathbb{\tilde{E}}_{t}[X]-\mathbb{\tilde{E}}_{t}%
[Y]\leq \mathbb{\hat{E}}_{t}[X-Y]$.
\end{description}
\end{proposition}

\begin{proof}
The proof is similar to Proposition \ref{pro4.4}, we omit it.
\end{proof}

The following example shows that the extension to $L_{G}^{1_{\ast}^{\ast}%
}(\Omega)$ is impossible for some nonlinear expectations.

\begin{example}
Define $\mathbb{\tilde{E}}_{t}[X]=-\mathbb{\hat{E}}_{t}[-X]$ for each $X\in
L_{ip}(\Omega)$. It is easy to check that $\mathbb{\tilde{E}}_{t}[\cdot]$
satisfies the above properties. We consider%
\[
X_{n}=-I_{\{ \bar{\sigma}^{2}-\frac{1}{n}<\langle B\rangle_{1}<\bar{\sigma
}^{2}\}}\text{, }\tilde{X}_{n}=0.
\]
It is easy to verify that $X_{n}$, $\tilde{X}_{n}\in L_{G}^{1^{\ast}}(\Omega)$
and $X_{n}\uparrow0$, $\tilde{X}_{n}\uparrow0$. But $\mathbb{\tilde{E}}%
[X_{n}]=-1$ and $\mathbb{\tilde{E}}[\tilde{X}_{n}]=0$ for $n\geq1$.
\end{example}

\renewcommand{\refname}{\large References}{\normalsize \ }

\end{document}